\documentclass[reqno,10pt]{amsart}
\usepackage{tikz}
\usepackage{cite}
\usepackage{amsmath,mathrsfs,amssymb,amsthm}
\usepackage{graphicx}
\usepackage{caption}
\usepackage{appendix}
\usepackage{geometry}
\usepackage{listings}
\usepackage{pifont}
\usepackage{float}
\usepackage{hyperref}
\usepackage{enumitem}
\usepackage{graphics}
\usepackage{cleveref}
\numberwithin{equation}{section}
\usepackage{verbatim}
\let\xi\xi
\allowdisplaybreaks[1]

\newtheorem{thm}{Theorem}[section]
\newtheorem{lem}[thm]{Lemma}

\newtheorem{remark}[thm]{Remark}

\begin{document}
	
\title[ASYMPTOTIC  STABILITY OF PRANDTL EQUATIONS WITH SUCTION]{ASYMPTOTIC STABILITY OF STEADY PRANDTL EQUATIONS WITH SUCTION}

\author[Y. H. Li]{Yonghao Li}
\address[Y. H. Li]{Academy of Mathematics and Systems Science, Chinese Academy of Sciences, Beijing 100190, China; School of Mathematical Sciences, University of Chinese Academy of Sciences, Beijing 100049, China.}
\email{liyonghao@amss.ac.cn}
	
\author[Y. Wang]{Yong Wang}
\address[Y. Wang]{Academy of Mathematics and Systems Science, Chinese Academy of Sciences, Beijing 100190, China; School of Mathematical Sciences, University of Chinese Academy of Sciences, Beijing 100049, China.}
\email{yongwang@amss.ac.cn}

\begin{abstract}
Despite its physical importance, the stability of Prandtl boundary layers under wall suction remains relatively unexplored in rigorous PDE analysis. In this paper, we establish the global-in-$x$ asymptotic stability of a suction profile. For any prescribed algebraic decay rate in $x$, we prove that sufficiently small inflow perturbations with sufficiently rapid spatial decay yield downstream decay at that rate. The von Mises transformation reduces the Prandtl system to a degenerate parabolic equation with a favorably directed convection term. We first identify the downstream decay mechanism in an associated uniformly parabolic problem, then control the boundary degeneracy using a sharp Hardy-type estimate. Together, these arguments yield stability and decay for the original Prandtl system.
\end{abstract}
	
	
	\date{\today}
	\maketitle
	
	\setcounter{tocdepth}{1}
	\tableofcontents
	
	\thispagestyle{empty}

\section{Introduction}

\subsection{Background}
Boundary-layer separation and hydrodynamic instability are major concerns in viscous wall-bounded flows. These phenomena increase drag, reduce aerodynamic efficiency, and can even lead to catastrophic flow breakdown in aerospace and mechanical systems. As discussed in Schlichting-Gersten \cite{Schlichting} (see page~48), wall suction is an effective means of suppressing boundary-layer separation. It removes low-momentum fluid from the boundary layer
through narrow slits in the wall. Sufficiently strong suction can eliminate separation entirely. Wall suction was first employed by L.~Prandtl in his seminal 1904 work on boundary layers over a circular cylinder.

This paper focuses on boundary layer flows on the domain $\Omega = \mathbb{R}^+ \times \mathbb{R}^+$ governed by the  Prandtl system:
\begin{equation}\label{generalprandtl}
	\begin{cases}
		uu_{x}+vu_{y}-u_{yy}+p'(x)=0,\\
		u_{x}+v_{y}=0,\\
		u(0,y)=u_{0}(y),\quad u|_{y=0}=0,\quad v|_{y=0}=-a,\quad \lim\limits_{y\to\infty} u=U(x),
 	\end{cases}
\end{equation}
where the constant $a>0$ characterizes the strength of wall suction, $(u,v)$ denotes the two-dimensional velocity field, $p(x)$ represents the streamwise pressure distribution, and $U(x)$ is the outer inviscid free-stream velocity.

The pressure and free-stream velocity satisfy the Bernoulli relation
\begin{equation*}
	2p(x)+U^{2}(x)=\text{constant}.
\end{equation*}
To simplify the analysis without losing essential structural features of boundary layers, we set the far-field speed identically equal to unity, $U\equiv 1$, which forces $p'(x)\equiv 0$. Under this normalization, the original system \eqref{generalprandtl} reduces to the homogeneous pressureless Prandtl equation
\begin{equation}\label{prandtl}
	\begin{cases}
		uu_{x}+vu_{y}-u_{yy}=0,\\
		u_{x}+v_{y}=0,\\
		u(0,y)=u_{0}(y),\quad u|_{y=0}=0, \quad v|_{y=0}=-a,\quad \lim\limits_{y\to\infty}u=1.
	\end{cases}
\end{equation}

A standard and powerful tool for analyzing the Prandtl system is
the von Mises transformation, which reduces the system to a
degenerate parabolic equation in the streamwise coordinate $X$
and the stream-function coordinate $\xi$. The transformation is
defined by
\begin{equation}\label{von}
	X=x,\qquad
	\xi=\xi(x,y;u)=\int_0^y u(x,s)\,ds.
\end{equation}
Using the incompressibility condition and the wall condition
$v(x,0)=-a$, we obtain
\[
\frac{\partial \xi}{\partial y}=u,\qquad
\frac{\partial \xi}{\partial x}=-(v+a),\qquad
\xi(x,0;u)=0.
\]
Since $u>0$ for $y>0$ and $u(x,y)\to 1$ as $y\to\infty$,
the mapping $(x,y)\mapsto(X,\xi)$ is a global diffeomorphism
from $D$ onto $G$, where
\[
D=\{0<x<L,\ 0<y<\infty\}\quad \mbox{and}\quad  G=\{0<X<L,\ 0<\xi<\infty\}.
\]
The inverse transformation is written as
\begin{equation}\label{invervon}
x=X,\qquad y=y(X,\xi;u).
\end{equation}

We introduce the squared streamwise velocity as the new unknown:
\begin{equation}\label{unknown}
	w(X,\xi;u)=u^{2}\bigl(X,y(X,\xi;u)\bigr).
\end{equation}
Then the system \eqref{prandtl} reduces to the following
quasilinear degenerate parabolic problem:
\begin{equation}\label{vonprandtl}
	\begin{cases}
		w_X-a w_\xi-\sqrt{w}\,w_{\xi\xi}=0,\\
		w(X,0)=0,\quad
		\displaystyle\lim_{\xi\to\infty}w(X,\xi)=1,\\
		w(0,\xi)=w_0(\xi),\quad
		\displaystyle
		w_0\left(\int_0^y u_0(s)\,ds\right)
		=\bigl(u_0(y)\bigr)^2.
	\end{cases}
\end{equation}
Oleinik \cite{Oleinik} pioneered the existence and uniqueness
theory for the Prandtl equations using the von Mises
transformation. The global existence and uniqueness result
for classical solutions that we use, together with its precise
hypotheses, is stated in Theorem~\ref{oleinik}.

\begin{thm}\label{oleinik}\textup{(}\textbf{Oleinik}\textup{)}
	Assume that the initial data $ u_{0}(y) $ satisfies for some $\alpha\in(0,1):$
	\begin{equation}\label{monotony}
		\begin{aligned}
			&u_{0}(y)\in C^{2,\alpha}([0,\infty)),\quad u_{0}(0)=0,\quad u_{0}(y)>0 \text{ for all } y>0,\\
			&u_{0}'(0)>0,\quad u_{0}''(y)+au_{0}'(y)=O(y^{2}) \text{ as y $\to$ }0,\quad \lim_{y\to\infty}u_{0}(y)=1.
		\end{aligned}
	\end{equation}
	Then the Prandtl equation \eqref{prandtl} admits a unique global classical velocity field $ (u,v) $ obeying
	\begin{equation*}
		u(x,y)>0\ \text{for}\ y>0,\quad u\in C^{1}(\mathbb{R}^{+}\times \mathbb{R}^{+}),\quad v, u_{yy}\in C^{0}(\mathbb{R}^{+}\times\mathbb{R}^{+}),
	\end{equation*}
	and for any finite  $L>0$, there exists a constant $y_0>0$ such that,
	\begin{itemize}
		\item[(1)] $ u, u_{y}, u_{yy} $ are uniformly bounded and continuous in $ [0,L]\times \mathbb{R}^{+},$
		\item[(2)] $ u_{x}, v, v_{y} $ are locally bounded  and continuous in $ [0,L]\times\mathbb{R}^{+},$
		\item[(3)] $ \frac{\partial u}{\partial y}>m>0$ in $[0,L]\times [0,y_0]$, where $m$ is a positive  constant depending on $y_0$. 
	\end{itemize}
\end{thm}

\subsection{Related works}
The study of the Prandtl boundary layer equations dates back to
Prandtl's pioneering work \cite{Prandtl}; see also
\cite{E,Oleinik,Schlichting} for systematic treatments of the
theory. For monotone velocity profiles,
Oleinik-Samokhin~\cite{Oleinik} developed a classical
well-posedness theory for the steady Prandtl equations using
the von Mises transformation and the maximum principle.
In particular, their theory yields global-in-$x$ solutions under a
favorable pressure gradient, subject to the requisite assumptions
on the data. Higher regularity estimates for local solutions
were subsequently established in \cite{Guo-Iyer-CMP}, while
Wang-Zhang~\cite{WangZhang2021} proved global $C^\infty$
regularity for steady Prandtl solutions under a favorable
pressure gradient.

An important problem in the theory of the steady Prandtl equation is the $x$-asymptotic behavior of its solutions. In the
zero-pressure-gradient case, subject to the impermeable-wall condition
\[
u|_{y=0}=v|_{y=0}=0,
\]
the most distinguished solution is the self-similar Blasius profile
\[
\tilde{u}(x,y)
=
f'\left(\frac{y}{\sqrt{x+1}}\right),
\]
where $f$ solves
\[
\begin{cases}
	f''' + \dfrac12 f f'' =0,
	& \text{in }\mathbb{R}^{+},\\[2mm]
	f(0)=f'(0)=0,
	\qquad& f'(\infty)=1.
\end{cases}\]

Serrin~\cite{Serrin} first proved that the Blasius profile is
a self-similar attractor for Oleinik solutions. More precisely,
using the von Mises transformation and a maximum-principle
argument, he showed that
\[
\lim_{x\to\infty}
\|u(x,\cdot)-\tilde{u}(x,\cdot)\|_{L^\infty_y}=0.
\]
This result established asymptotic convergence without an
explicit convergence rate.

The first quantitative global-in-$x$ stability result for the
Blasius profile was obtained by Iyer~\cite{Iyer}. For sufficiently
small  perturbations in suitable weighted Sobolev
spaces, he derived weighted energy estimates in von Mises
coordinates and obtained, up to a small loss in the decay exponent,
\[
\|u(x,\cdot)-\tilde{u}(x,\cdot)\|_{L^\infty_y}
\lesssim (1+x)^{-\frac12+\delta},
\qquad \delta>0.
\]
His argument also yields decay estimates for higher-order
derivatives of the transformed perturbation. Wang-Zhang~\cite{Z. Zhang}
subsequently removed the smallness assumption under a stronger
localization condition at spatial infinity, using the comparison
principle and carefully constructed barrier functions to treat
general Oleinik solutions subject to that condition.

More recently, Jia-Lei-Yuan~\cite{Lei} identified a mechanism
leading to faster decay and established the sharp estimate
\[
\|u(x,\cdot)-\tilde{u}(x,\cdot)\|_{L^\infty_y}
\lesssim (1+x)^{-1}
\]
under suitable weighted assumptions on the initial perturbation.
Their proof combines weighted $L^1$ estimates, refined energy
inequalities, and a finite iteration scheme that successively
improves the decay rate. The analysis uses weighted $L^1$ control
of the perturbation and accounts for the boundary degeneracy
of the transformed equation. They also constructed a family of exact solutions showing that
the rate $(1+x)^{-1}$ cannot, in general, be improved. Thus,
this rate is optimal within the class considered in their work.
Their result also covers certain non-small initial data satisfying
polynomial localization conditions.

The $x$-asymptotic stability theory in the absence of wall suction has recently been extended beyond the
zero-pressure-gradient Blasius profile. Iyer\cite{Iyer2026FalknerSkan} studied the family of
favorable Falkner--Skan profiles and established quantitative
stability and scattering estimates for the stationary Prandtl
equations. These results clarify how a
favorable pressure gradient influences the downstream stabilization
of Prandtl solutions. 

In the opposite direction, an adverse pressure
gradient may lead to boundary-layer separation. The formation and
local structure of separation for the steady Prandtl equation were
studied by Dalibard-Masmoudi\cite{DM} and Shen-Wang-Zhang\cite{ShenWangZhang2021}. Together
with the favorable-pressure regularity result
\cite{WangZhang2021}, these works illustrate the fundamentally
different behaviors produced by favorable and adverse pressure
gradients.

\subsection{Formulation and main results} In the present paper, we consider the classical suction profile (see \cite{Schlichting})
\begin{align}
\bar{u}(y)=1-e^{-ay},\quad \bar{v}\equiv -a
\end{align}
which indeed solves the  Prandtl system \eqref{prandtl} exactly. Following \cite{Iyer}, we introduce a scaled perturbation variable $\phi$ via
\begin{align}
	\epsilon \phi := u^2(X,y(X,\xi;u))-\bar{u}^2(X,y(X,\xi;\bar{u})),
\end{align}
then  $\phi$ satisfies a  degenerate parabolic equation
\begin{equation}\label{eqphi}
	\begin{cases}
		\phi_{X}-a\phi_\xi-u\phi_{\xi\xi}+A\phi=0,\quad \text{in}\ \mathbb{R}^{+}\times\mathbb{R}^{+},\\
		\phi(0,\xi)=\phi_{0}(\xi),\quad \phi(X,0)=\phi(X,\infty)=0,\\
		A(X,\xi)=\displaystyle\frac{-2\bar{u}_{yy}}{\bar{u}(u+\bar{u})}\bigg|_{(X,\xi)},
	\end{cases}
\end{equation}
where the initial perturbation is given by
\begin{align}\label{phi_0}
	\phi_{0}(\xi)=\frac{u_{0}^2\big(y(\xi;u_0)\big)-\bar{u}^2\big(y(\xi;\bar{u})\big)}{\epsilon}.
\end{align}
We also denote
\begin{align}
	\epsilon\rho(X,\xi) :=u\big(X,y(X,\xi;u)\big)-\bar{u}\big(y(\xi;\bar{u})\big).
\end{align}

For later use, we define
\begin{equation}\label{Y0,Y1}
	\begin{aligned}
		\|\phi\|_{\mathbf{Y}_0}^2
		&= 
		\left\|
		\frac{\phi_\xi}{u^{\lambda-\frac{1}{2}}}
		\right\|_{L^2_{X, \xi }}^2+ \ \sup_{X \ge 0}
		\left\|
		\frac{\phi}{u^{\lambda}}
		\right\|_{L^2_{\xi}}^2, \\[1ex]
		\|\phi\|_{\mathbf{Y}_1}^2
		&= 
		\left\|
		\frac{\phi_X}{\sqrt{u}}
		\right\|_{L^2_{X , \xi }}^2 + \ \sup_{X \ge 0}
		\left\|
		\phi_\xi
		\right\|_{L^2_{\xi }}^2 ,
	\end{aligned}
\end{equation}
and
\begin{align}\label{Y}
	\|\phi\|_{\mathbf{Y}}^2 = \|\phi\|_{\mathbf{Y}_0}^2 + \|\phi\|_{\mathbf{Y}_1}^2,
\end{align}
where $\lambda$ $\in$ $[1,\frac{3}{2})$.

For initial data, we introduce the corresponding initial norms:
\begin{equation}\label{F0,0}
	\begin{aligned}
		\|\phi_{0}\|_{\mathbf{Y}_{0,0}}^{2}
		&= \int \frac{|\phi_{0}(\xi)|^{2}}{u_{0}^{2\lambda}} \mathrm{d}\xi,\\
		\|\phi_{0}\|_{\mathbf{Y}_{1,0}}^{2}
		&= \int \big|\phi_{0}'(\xi)\big|^{2} \mathrm{d}\xi,\\
		\|\phi_{0}\|_{\mathbf{Y}_{2,0}}^{2}&=\int_{0}^{+\infty}\frac{|\phi_{X}(0,\xi)|^2}{u_0} d\xi,\\
		\|\phi_{0}\|_{\mathbf{Y}_{in}}^{2}
		&=\|\phi_{0}\|_{\mathbf{Y}_{0,0}}^{2}+	\|\phi_{0}\|_{\mathbf{Y}_{1,0}}^{2}+
		\|\phi_{0}\|_{\mathbf{Y}_{2,0}}^{2}.
	\end{aligned}
\end{equation}
We remark that $\phi_{0}'(\xi)$ is given by
\begin{equation}
		\begin{aligned}
		&	\epsilon \phi_{0}'(\xi)
			= 2(\partial_{y}u_0)\big(y(\xi;u_0)\big) - 2(\partial_{y}\bar{u})\big(y(\xi;\bar{u})\big).
		\end{aligned}
\end{equation}
And by the compatibility condition, $\phi_{X}(0,\xi)$ is given by
\begin{equation}
	\begin{aligned}
		\epsilon\phi_{X}(0,\xi)
	= \big(2a u_{0y}(y(\xi;u_0)) + 2u_{0yy}(y(\xi;u_0))\big)
	- \big(2a \bar{u}_{y}(y(\xi;\bar{u})) + 2\bar{u}_{yy}(y(\xi;\bar{u}))\big).
	\end{aligned}
\end{equation}

Our main result is:
\begin{thm}\label{Main}
Let $\phi$ be the solution of \eqref{eqphi}. Assume \eqref{monotony}. There exists a sufficiently small
	constant $\epsilon>0$ such that, for every $a>0$, if the
	initial datum $u_0(y)$ in physical coordinates satisfies
	\begin{align}\label{condition}
		\|u_0-\bar{u}\|_{L_y^1}
		+ \|u_0-\bar{u}\|_{H_y^1}
		+ \|\partial_y(u_0-\bar{u})\|_{L_y^\infty}
		\leq \frac{a}{(1+a)^2}\epsilon,
	\end{align}
then the following estimates hold:
	\begin{align}
		\|\phi\|_{\mathbf{Y}_0}
		&\lesssim
		\|\phi_0\|_{\mathbf{Y}_{0,0}},
		\label{esimate in Y0}\\
		\|\phi\|_{\mathbf{Y}_1}
		&\lesssim
		a\|\phi_0\|_{\mathbf{Y}_{0,0}}
		+ \|\phi_0\|_{\mathbf{Y}_{1,0}}.
		\label{estimate in Y1}
	\end{align}
\end{thm}

\begin{remark}
	The $a$-dependent constraint \eqref{condition} is critical to guarantee uniform asymptotic equivalence $u\sim\bar{u}$ with respect to the suction parameter $a>0$ in von Mises coordinates; detailed justifications can be found in Lemma \ref{u0simbaru}.
\end{remark}
\begin{remark}
	It is worth noting that the condition \eqref{condition} is sufficient to guarantee that $\|\phi_{0}\|_{\mathbf{Y}_{0,0}}$ and $\|\phi_0\|_{\mathbf{Y}_{1,0}}$ are finite, see Lemma \ref{phi0} in Appendix A for details of proof.
\end{remark}

Our second result, Theorem~\ref{Main2}, gives explicit pointwise
convergence rates of $u$ to $\bar{u}$ in von Mises coordinates.

\begin{thm}\label{Main2}
	Assume \eqref{monotony}. There exists a sufficiently small
	constant $\epsilon>0$ such that, for every $a>0$ and every
	positive integer $m$, if the initial datum $u_0(y)$ in physical
	coordinates satisfies \eqref{condition} and
	\begin{align}\label{weight}
		\|(1+\xi)^{m/2}\phi_0\|_{\mathbf{Y}_{0,0}}<\infty,
	\end{align}
	then, for all $X>0$, the following estimate holds:
	\begin{equation}
		\begin{aligned}
			&\sup_{\xi\geq 0}
			\left|
			u\bigl(X,y(X,\xi;u)\bigr)
			-\bar{u}\bigl(y(\xi;\bar{u})\bigr)
			\right|\\
			&\leq
			\epsilon C_{m,a}
			\left(
			\|(1+\xi)^{m/2}\phi_0\|_{\mathbf{Y}_{0,0}}
			+\|\phi_0\|_{\mathbf{Y}_{1,0}}
			\right)
			(1+X)^{-\frac{m}{2}},
		\end{aligned}
	\end{equation}
	where $C_{m,a}$ depends only on $m$ and $a$.
\end{thm}

\begin{remark}
	Under the remaining assumptions of Theorem~\ref{Main2},
	the weighted condition \eqref{weight} follows from the
	following condition on the initial datum in physical coordinates:
	\begin{align}\label{initcond4}
		\|y^{m/2}(u_0-\bar{u})\|_{L_y^2}<\infty,
	\end{align}
see Lemma \ref{phi0m} in Appendix A for details of proof.
\end{remark}

Our third result, Theorem~\ref{Main3}, transfers the convergence
estimate in von Mises coordinates to the original physical
coordinates, yielding an explicit decay rate for the difference
$u(x,y)-\bar{u}(y)$.

\begin{thm}\label{Main3}
	Fix $a>0$ and a positive integer $m$.
	Assume \eqref{monotony}, \eqref{condition}, and \eqref{weight}.
	There exists a sufficiently small constant $\delta_{a,m}>0$
	such that, if
	\begin{align}\label{pho initial}
		\|(1+\xi)^{m/2}\phi_0\|_{\mathbf{Y}_{0,0}}
		+ \|\phi_0\|_{\mathbf{Y}_{\mathrm{in}}}
		\leq \delta_{a,m},
	\end{align}
	then, for all $x>0$, it holds that 
	\begin{equation}
		\begin{aligned}
			\sup_{y\geq 0}|u(x,y)-\bar{u}(y)|
			\leq &
			\epsilon C_{m,a}
			\left(
			\|(1+\xi)^{m/2}\phi_0\|_{\mathbf{Y}_{0,0}}
			+ \|\phi_0\|_{\mathbf{Y}_{\mathrm{in}}}
			\right)
			(1+x)^{-\frac{m}{2}},
		\end{aligned}
	\end{equation}
	where $C_{m,a}$ depends only on $m$ and $a$.
\end{thm}

\begin{remark}
	For each positive integer $m$, the theorem gives convergence
	to the suction profile at the algebraic rate $(1+x)^{-m/2}$,
	provided that the initial perturbation satisfies the
	corresponding weighted smallness condition. Thus, arbitrarily
	high algebraic decay rates are obtained under suitable
	assumptions on the initial perturbation.
	By comparison, the optimal convergence rate in the
	zero-suction setting considered in \cite{Lei} is
	$(1+x)^{-1}$. This comparison highlights the enhanced
	downstream decay associated with wall suction.
\end{remark}

\begin{remark}
	There exists a sufficiently small constant $\eta_{a,m}>0$
	such that, if the initial datum $u_0(y)$ in physical coordinates
	satisfies
	\begin{align}\label{initcond5}
		\|u_0-\bar{u}\|_{L_y^1}
		+ \|u_0-\bar{u}\|_{H_y^2}
		+ \|y^{m/2}(u_0-\bar{u})\|_{L_y^2}
		\leq \eta_{a,m},
	\end{align}
	then \eqref{pho initial} holds, see Lemma \ref{F2,0} in Appendix~A for details of proof.
\end{remark}

\subsection{Main idea}
Now we explain the main idea in the present paper.
For simplicity, we consider the linearized equation of \eqref{eqphi}, which reads as \begin{align}\label{idea1}
		\phi_{X}-a\phi_\xi-\bar{u}\phi_{\xi\xi}-\frac{\bar{u}_{yy}}{\bar{u}^2}\phi=0.
\end{align}

1. Applying the basic energy estimate on \eqref{idea1}, one can easily derive \begin{align}\label{idea112}
	\sup_{X \ge 0}
	\left\|
\phi
	\right\|_{L^2_{\xi}}^2
	+
	\left\|
	\sqrt{\bar{u}}\phi_{\xi}
	\right\|_{L^2_{X, \xi }}^2-\int_{0}^{\infty}\int_{0}^{\infty} \frac{\bar{u}_{yy}}{\bar{u}^2}\phi^2\lesssim \|\phi_{0}\|_{L_{\xi}^2}^2.
\end{align}

However, due to the presence of suction,  we encounter difficulties when controlling $\sup_{X \geq0}\|\phi_{\xi}\|_{L^{2}_{\xi}}^{2}$.
Specifically, multiplying  \eqref{idea1} by $\frac{1}{\bar{u}}\phi_{X}$, one derives 
	\begin{align}\label{idea2}
	\int_{0}^{\infty} \frac{\phi_{X}^2}{\bar{u}}
	+\frac12 \frac{d}{d X}\int_{0}^{\infty} \phi_{\xi}^2
	=a\int_{0}^{\infty} \frac{\phi_{\xi}\phi_{X}}{\bar{u}}
	+\int_{0}^{\infty} \frac{\bar{u}_{yy}}{\bar{u}^3}\phi\phi_{X}.
\end{align}
	It is clear that
\begin{align}\label{idea3}
\begin{split}
	a\left|\int_{0}^{\infty}\int_{0}^{\infty} \frac{\phi_{\xi}\phi_{X}}{\bar{u}}\right|
&\leq \frac14 \int_{0}^{\infty}\int_{0}^{\infty} \frac{\phi_X^2}{\bar{u}}
+a^2\int_{0}^{\infty}\int_{0}^{\infty} \frac{\phi_\xi^2}{\bar{u}},\\
\bigg|\int_{0}^{\infty}\int_{0}^{\infty} \frac{\bar{u}_{yy}}{\bar{u}^3}\phi\phi_{X}\bigg|
&\leq     \frac14 \int_{0}^{\infty}\int_{0}^{\infty} \frac{\phi_X^2}{\bar{u}}+
\int_{0}^{\infty}\int_{0}^{\infty}\frac{|\bar{u}_{yy}|^2}{\bar{u}^{5}}\phi^2.  
\end{split}
\end{align}
Intuitively, without higher‑order derivative  estimates, the terms
\begin{align}\label{1.31}
	\int_{0}^{\infty} 	\int_{0}^{\infty}\frac{\phi_{\xi}^{2}}{\bar{u}} \quad \mbox{and}\quad 
	\int_{0}^{\infty}\int_{0}^{\infty}\frac{|\bar{u}_{yy}|^2}{\bar{u}^{5}}\phi^2,
\end{align}
on RHS of \eqref{idea3} are difficult to bound.

If we apply the same quotient estimate introduced by Iyer \cite{Iyer}, then one can only get  \begin{align}
		\left\|
		\phi_{\xi}
		\right\|_{L^2_{X, \xi }}^2+\	\sup_{X \ge 0}
		\left\|
		\frac{\phi}{\sqrt{\bar{u}}}
		\right\|_{L^2_{\xi}}^2\leq C\bigg\|\frac{\phi_{0}}{\sqrt{u_0}}\bigg\|_{L_{\xi}^2}^2,
\end{align}
which is not enough to control the first term in \eqref{1.31}. Instead, we consider a sharper quotient estimate, that is, multiplying \eqref{idea1} by  $\frac{1}{\bar{u}^{2\lambda}} \phi$ with $\lambda\in[1,\frac32)$, one has 
\begin{align}\label{idea114}
	\frac{1}{2}\frac{d}{dX}\int_{0}^{\infty}\frac{\phi^2}{\bar{u}^{2\lambda}}+\int \frac{\phi_\xi^2}{\bar{u}^{2\lambda-1}}
	&\leq  (2\lambda-1)\bigg\|\frac{\phi_\xi}{\bar{u}^{\lambda-\frac{1}{2}}}\bigg\|_{L_{\xi}^2}
	\bigg\|\frac{\bar{u}_y\phi}{\bar{u}^{\lambda+\frac{3}{2}}}\bigg\|_{L_{\xi}^2}+ (\lambda-1)\bigg\|\frac{|\bar{u}_{yy}|^{\frac{1}{2}}\phi}{\bar{u}^{\lambda+1}}\bigg\|_{L_{\xi}^2}^2.
\end{align} 
It is noted that 
the most difficult term on RHS of \eqref{idea114} is $\left\| \frac{\phi\bar{u}_y\chi_{\frac{\delta}{a}}}{\bar{u}^{\lambda+\frac{3}{2}}} \right\|_{L_\xi^2}.$
By using a sharp Hardy-type inequality established in Lemma \ref{Lower base}, we can obtain	\begin{align}\label{idea113}
		\left\|\frac{\phi\bar{u}_y\chi_{\frac{\delta}{a}}}{\bar{u}^{\lambda+\frac{3}{2}}}\right\|_{L_\xi^2}
	\leq \frac{1}{C_2(\delta)^{2}}\frac{4}{2\lambda+1} \Bigg(
	\left\|  \frac{\bar{u}_y \phi\chi_{\frac{\delta}{a}}}{\bar{u}^{\lambda+\frac{1}{2}}}\right\|_{L_{\xi}^2}
	+ \left\|  \frac{ \phi_\xi\chi_{\frac{\delta}{a}}}{\bar{u}^{\lambda-\frac{1}{2}}}\right\|_{L_{\xi}^2}
	+\frac{C}{\delta} \left\| \frac{\bar{u}_y \phi}{\bar{u}^{\lambda-\frac{1}{2}}} \right\|_{L_{\xi\geq \frac{\delta}{a}}^2}
	\Bigg),
\end{align}
where
\begin{align}\label{1.36}
\lim_{\delta\to 0}C_2(\delta)=\sqrt{2} .
\end{align}
Due to  $\lambda\in[1,\frac{3}{2})$ and \eqref{1.36}, it holds that 
\begin{align*}
	\frac{2\lambda-1}{C_{2}(\delta)^2}\frac{4}{2\lambda+1}<1\quad \mbox{for}\,\, \delta\ll1,
\end{align*}
thus
\begin{align}
 (2\lambda-1)\bigg\|\frac{\phi_\xi}{\bar{u}^{\lambda-\frac{1}{2}}}\bigg\|_{L_{\xi}^2}
\bigg\|\frac{\bar{u}_y\phi}{\bar{u}^{\lambda+\frac{3}{2}}}\bigg\|_{L_{\xi}^2}
&\leq \frac{2\lambda-1}{C_{2}(\delta)^2}\frac{4}{2\lambda+1} \bigg\|\frac{\phi_\xi}{\bar{u}^{\lambda-\frac{1}{2}}}\bigg\|_{L_{\xi}^2}^2 + \cdots,
\end{align}
which implies that $\frac{2\lambda-1}{C_{2}(\delta)^2}\frac{4}{2\lambda+1} \bigg\|\frac{\phi_\xi}{\bar{u}^{\lambda-\frac{1}{2}}}\bigg\|_{L_{\xi}^2}^2$ can be absorbed by the second term on LHS of \eqref{idea114}. We remark that the sharp Hardy-type inequality Lemma \ref{Lower base} plays an essential role in the  above argument. For the second term  of \eqref{1.31} and the second term on RHS of \eqref{idea114}, we need to use the sharp Hardy-type estimates twice to bound these terms. Then one can get the following sharper quotient estimate
\begin{align}\label{1.37}
\left\|
	\frac{\phi_\xi}{u^{\lambda-\frac{1}{2}}}
	\right\|_{L^2_{X, \xi }}^2+ \ \sup_{X \ge 0}
	\left\|
	\frac{\phi}{u^{\lambda}}
	\right\|_{L^2_{\xi}}^2 \lesssim  \| \phi_0\|_{\mathbf{Y}_{0,0}}^{2},
\end{align}
which, together with \eqref{idea2}, yields that 
\begin{align}\label{1.38}
\left\|
	\frac{\phi_X}{\sqrt{u}}
	\right\|_{L^2_{X , \xi }}^2 + \ \sup_{X \ge 0}
	\left\|
	\phi_\xi
	\right\|_{L^2_{\xi }}^2 \lesssim a^2\| \phi_0 \|_{\mathbf{Y}_{0,0}}^{2}+
	\| \phi_0 \|_{\mathbf{Y}_{1,0}}^{2}.
\end{align}
We refer to Section \ref{subsection2.2} for the details of proof.

\smallskip


2. Due to the effect of suction with $a>0$, and using \eqref{1.37}-\eqref{1.38}, and then applying the weighted parabolic argument trading spatial weights for temporal decay, we can obtain the following key estimates:
	\begin{align}
		\sup_{X \geq0}
		\int_0^\infty X^{m} \frac{\phi^2}{u^{2}}
		+ \int_0^\infty\int_0^\infty X^{m} \frac{\phi_{\xi}^2}{u}
		&\leq   C_{m,a} \|(1+\xi)^{\frac{m}{2}}\phi_{0}\|_{\mathbf{Y}_{0,0}}^{2},\label{1.39}\\
	\sup_{X \geq0}
	\int_0^\infty X^m \phi_{\xi}^2 \,
	+ \int_0^\infty\int_0^\infty X^m \frac{\phi_{X}^2}{u}
	&\leq C_{m,a}\|(1+\xi)^{\frac{m}{2}}\phi_{0}\|_{\mathbf{Y}_{0,0}}^{2}+\|\phi_{0}\|_{\mathbf{Y}_{1,0}}^2,\label{1.40}
\end{align}
where $m$ is a nonnegative integer, and $C_{m,a}$ is a constant depending only on $m$ and $a$.
Then it follows from \eqref{1.39}-\eqref{1.40} and the  Sobolev embedding theory that
\begin{equation}
	\begin{aligned}
			\|\phi\|_{L_{\xi}^\infty}& \le \|\phi\|_{L_{\xi}^2}^{\frac12} \|\phi_\xi\|_{L_{\xi}^2}^{\frac12}\le C_{m,a} (1+X)^{-\frac{m}{2}} \left(\big\| (1+\xi)^{\frac{m}{2}} \phi_0 \big\|_{L_{\xi}^2}+ \| \phi_{0\xi} \|_{L_{\xi}^2}\right).
	\end{aligned}
\end{equation}
We refer to Section \ref{decay von mises} for the detailed proof.

\smallskip 

3. To obtain the decay estimate in the original physical coordinates, however, the available decay estimates for $\|\phi\|_{L_{\xi}^2}$ and $\|\phi_{\xi}\|_{L_{\xi}^2}$ are insufficient to yield pointwise decay in physical coordinates. Motivated by \cite{Lei}, we first derive a decay estimate for $\left\| \frac{\phi_X}{\sqrt{u}} \right\|_{L_{\xi}^2}$, i.e.,
\begin{align*}
	&\sup_{ X \geq 0}\Big\{(1+X)^{m}\int_0^\infty \frac{\phi_{X}^2}{u} \Big\}
	+  \int_0^\infty  \int_0^\infty (1+X)^{m} |\phi_{X\xi}|^2  
	\leq C_{a,m}  \left(\|(1+\xi)^{\frac{m}{2}}\phi_{0}\|_{\mathbf{Y}_{0,0}}^2+\|\phi_{0}\|_{\mathbf{Y}_{in}}^2\right),
\end{align*}
which further yields that
	\begin{equation*}
	\begin{aligned}
\|\phi_\xi\|_{L_{\xi}^4}
		\leq C_{m,a}\left( \|(1+\xi)^{\frac{m}{2}}\phi_{0}\|_{\mathbf{Y}_{0,0}}+ \|\phi_{0}\|_{\mathbf{Y}_{in}}\right)(1+X)^{-\frac{m}{2}}.
	\end{aligned}
\end{equation*}	
With the help of above estimates, one can obtain 
\begin{align}
	\sup_{y\geq 0}|u(x,y)-\bar{u}(y)|
\leq &
\epsilon C_{m,a}
\left(
\|(1+\xi)^{m/2}\phi_0\|_{\mathbf{Y}_{0,0}}
+ \|\phi_0\|_{\mathbf{Y}_{\mathrm{in}}}
\right)
(1+x)^{-m/2},
\end{align}
We refer to Section \ref{decay physical} for the detailed proof.

\subsection{Notations.} We fix notation conventions used throughout the manuscript:
We adopt the standard comparison notation:
 	$A\lesssim B$ means $A\le cB$ for some absolute constant $c>0$, and $A\lesssim_\beta B$ means $A\le c'B$ for some constant $c'>0$ depending only on $\beta$. Similarly, $A\sim_\beta B$ and $A\sim B$ stand for mutual two-sided estimates $A\lesssim_\beta B\lesssim_\beta A$ and $A\lesssim B\lesssim A$, respectively.
  For brevity, integrals without explicit integration limits are abbreviated as
 	\begin{equation*}
 		\int \cdots := \int_{0}^{\infty}\cdots\,d\xi , \quad 
 		\int \int \cdots := \int_{0}^{\infty} \int_{0}^{\infty}\cdots\,d\xi \, dX.
 	\end{equation*}

\section{Global Existence}\label{YYYY}
For later use, we define the following cut-off function
	\begin{equation}\label{cut}
		\chi(\xi)=
		\begin{cases}
			1, & \xi\le 1,\\
			0, & \xi\ge 2.
		\end{cases}
	\end{equation}
	
For any $\alpha>0$, we denote
\begin{align}
\chi_{\alpha}(\xi):=\chi\!\left(\frac{\xi}{\alpha}\right).
\end{align}

	\subsection{Useful estimates}
\begin{lem}\label{baru}
	For every fixed $s>0$, the following estimates hold:
	\begin{equation}\label{equiva}
		\begin{aligned}
			&\bar{u}(y(\xi;\bar{u}))\sim_s 1, \quad \text{when} \quad \xi\geq \frac{s}{a},\\
			& y(\xi;\bar{u}) \sim_s \xi, \quad \text{when} \quad \xi\geq \frac{s}{a},\\
			&\bar{u}(y(\xi;\bar{u}))\sim_s ay(\xi;\bar{u})\sim_s \sqrt{a\xi}, \quad \text{when} \quad \xi\leq \frac{s}{a},\\
			&\lim_{\xi\to 0}\frac{\bar{u}(y(\xi;\bar{u}))}{\sqrt{a\xi}}=\sqrt{2}.
		\end{aligned}
	\end{equation}
\end{lem}

\begin{proof} Recall $\bar{u}(y)=1-\mathrm{e}^{-ay}$ and its associated stream-function
	\begin{align}\label{stream}
		\xi=\int_{0}^{y}\bar{u}(y')\,\mathrm{d}y'=y+\frac1a\bigl(\mathrm{e}^{-ay}-1\bigr).
	\end{align}
We divide the proof into three cases.

\smallskip

1. We consider the case for $\xi\ge \frac{s}{a}$. Then, from \eqref{stream}, it is clear that
\[
ay+\mathrm{e}^{-ay}-1\ge s.
\]
By strict monotonicity of the function $f(t)=t+\mathrm{e}^{-t}-1$, there exists a constant $C_1(s)$ depending only on $s$ such that
	\[
	ay\ge C_1(s) \quad \text{iff} \quad \xi \geq \frac{s}{a},
	\]
which yields immediately that
	\[
	\tilde{C}_1(s)\le \bar{u}(y)\le 1 \quad \text{for} \quad \xi\geq \frac{s}{a}.
	\]
	Thus we obtain
	\[
	\bar{u}\big(y(\xi;\bar{u})\big)\sim_s 1 \quad \text{for} \quad \xi\geq \frac{s}{a}.
	\]
	Since $\xi\geq \frac{s}{a}$ implies $y\geq \frac{C_1(s)}{a}$, it holds that
	\[
	\xi=\int_{0}^{y}\bar{u}(y')\,\mathrm{d}y'
	\ge \int_{\frac{C_1(s)}{4a}}^{y}\bar{u}(y')\,\mathrm{d}y'
	\ge \frac{\tilde{C}_1(s)}{2}y \quad \text{for} \quad y\geq \frac{C_1(s)}{a}.
	\]
which, together with $\bar{u}(y)\le 1$, yields that
	\[
	\frac{\tilde{C}_1(s)}{2}\,y \le \xi \le y.
	\]
	Hence we obtain
	\[
	y(\xi;\bar{u}) \sim_s \xi \quad \text{for}\quad \xi\geq \frac{s}{a}.
	\]
	
\smallskip

2. Next, we consider the case for $\xi\leq \frac{s}{a}$. It is clear that the constraint $\xi\leq \frac{s}{a}$ is equivalent to
	\begin{align*}
		ay\leq C_1(s),
	\end{align*}
	where $C_1(s)>0$ is a constant depending only on $s$.
It holds that
	\[
	\bar{u}(y)=ay-\frac{(ay)^2}{2}+o((ay)^2) \quad \text{as} \quad ay\to 0,
	\]
	which implies the asymptotic equivalence
	\begin{align}\label{Taylor}
		\bar{u}(y(\xi;\bar{u}))\sim_s ay(\xi;\bar{u}) \quad \text{for} \quad \xi\leq \frac{s}{a}.
	\end{align}
Then one has from \eqref{Taylor} that
	\[
	\xi
	=\int_{0}^{y}\bar{u}(y')\,\mathrm{d}y'
	\sim_s \int_{0}^{y} a y'\,\mathrm{d}y'
	= \frac12 a y^2 \quad \text{for} \quad ay\leq C_1(s),
	\]
	which implies
	\[
	ay(\xi;\bar{u}) \sim_s \sqrt{a\xi}.
	\]

3. 	Finally, we consider the limit $\eqref{equiva}_4$.
Denote $t = ay$, as $\xi\to 0$, one has from \eqref{stream} that
	\[
	a\xi = t + \mathrm{e}^{-t} - 1 = \tfrac{1}{2}t^2+ o(t^2),
	\]
which implies that 
	\[
	\lim_{\xi\to 0}
	\frac{\bar{u}\big(y(\xi;\bar{u})\big)}{\sqrt{a\xi}}
	=	\lim_{t\to 0} \frac{t+o(t)}{\sqrt{t^2/2+o(t^2)}} = \sqrt{2}.
	\]
Therefore the proof of Lemma \ref{baru} is completed.
\end{proof}

\begin{remark}\label{H,h}
	Denote
	\begin{align}
		H(\delta)=\sup_{\xi\leq \frac{\delta}{a}}\left|\frac{\bar{u}(y(\xi;\bar{u}))}{\sqrt{a\xi}}\right|, \qquad h(\delta)=\inf_{\xi\leq \frac{\delta}{a}}\left|\frac{\bar{u}(y(\xi;\bar{u}))}{\sqrt{a\xi}}\right|.
	\end{align}
	From Lemma \ref{baru}, thus we have
	\begin{equation}\label{2.6}
		\begin{aligned}
			& h(\delta)\sqrt{a\xi} \leq \bar{u}(y(\xi;\bar{u}))\leq H(\delta)\sqrt{a\xi}, \\
			& \lim_{\delta\to 0}H(\delta)=\lim_{\delta\to 0}h(\delta)=\sqrt{2}.
		\end{aligned}
	\end{equation}
We emphasize that the behavior \eqref{2.6} plays a crucial role when applying  Hardy's inequality in Section \ref{subsection2.2}.
\end{remark}

\begin{lem}\label{u0simbaru}
	Under the assumption  \eqref{condition}, for $(X,\xi)\in \mathbb{R}^{+}\times \mathbb{R}^{+}$, we have
	\begin{equation}\label{lyh1}
		\begin{aligned}
			|\phi(X,\xi)|&\lesssim \bar{u}^2\big(y(\xi;\bar{u})\big),\\
			|\rho(X,\xi)| &\lesssim \bar{u}\big(y(\xi;\bar{u})\big).
		\end{aligned}
	\end{equation}
	Since $\epsilon$ is sufficiently small, it holds that
	\begin{equation}
		u\big(X,y(X,\xi;u)\big) \sim \bar{u}\big(y(\xi;\bar{u})\big),
	\end{equation}
	over the domain $\mathbb{R}^{+}\times \mathbb{R}^{+}$.
\end{lem}

\begin{proof} We follow the argument in Lemma 2.2 of \cite{Lei}. And we perform a refined quantitative analysis to ensure that all constants in the resulting bounds are independent of $a$.	 

\smallskip

1. We first aim to prove
\begin{align}\label{initial0}
	\big|u_{0}^2\big(y(\xi;u_0)\big)-\bar{u}^2\big(y(\xi;\bar{u})\big)\big|\lesssim \epsilon\,\bar{u}^2\big(y(\xi;\bar{u})\big),
	\quad \forall \xi \geq0.
\end{align}
It is clear that
\begin{align*}
	\xi(y;u_0)=\int_{0}^{y}u_0(s)\mathrm{d}s =\xi(y;\bar{u})+\int_{0}^{y}\big[u_0(s)-\bar{u}(s)\big]\mathrm{d}s.
\end{align*}

\smallskip

1.1.	For $y\leq 1$, it follows from \eqref{condition} that
	\begin{equation}\label{initial1}
		\begin{aligned}
			\big|\xi(y;u_0)-\xi(y;\bar{u})\big|
			&\leq \int_{0}^{y} \big\| \partial_y(u_0-\bar{u}) \big\|_{L^\infty} s\mathrm{d}s \lesssim \frac{a\epsilon}{(1+a)^2} y^2,
		\end{aligned}
	\end{equation}
	where we have used $u_0(0)-\bar{u}(0)=0$.
	
For $y\geq 1$,  we have from \eqref{condition} that
	\begin{equation}\label{initial1.5}
		\begin{aligned}
			\big|\xi(y;u_0)-\xi(y;\bar{u})\big|\leq \big\| u_0-\bar{u} \big\|_{L^1}\lesssim \frac{a\epsilon}{(1+a)^2},
		\end{aligned}
	\end{equation}
which, together with \eqref{initial1}, yields that
	\begin{align}\label{initial2}
		\big|\xi(y;u_0)-\xi(y;\bar{u})\big|\lesssim\frac{a\epsilon}{(1+a)^2} \sigma^2(y),
	\end{align}
	where
	\begin{align}\label{sigma}
		\sigma(y)=
		\begin{cases}
			y, & \text{for} \quad y\leq 1,\\
			1, & \text{for} \quad y\geq 1.
		\end{cases}
	\end{align}
	
1.2. Since $\xi(y(\xi;u_0);u_0)=\xi$, it follows from \eqref{initial2} that
	\begin{align*}
		\xi-M\frac{a\epsilon}{(1+a)^2}\sigma^2\big(y(\xi;u_0)\big)\leq \xi\big(y(\xi;u_0);\bar{u}\big)\leq \xi+ M\frac{a\epsilon}{(1+a)^2}\sigma^2\big(y(\xi;u_0)\big),
	\end{align*}
	which yields that
	\begin{align*}
		\xi^{-1}\bigg(\xi-M\frac{a\epsilon}{(1+a)^2}\sigma^2\big(y(\xi;u_0)\big);\bar{u}\bigg)\leq y(\xi;u_0)\leq \xi^{-1}\bigg(\xi+M\frac{a\epsilon}{(1+a)^2}\sigma^2\big(y(\xi;u_0)\big);\bar{u}\bigg).
	\end{align*}
	As a consequence, we obtain
	\begin{equation}\label{inequality}
		\begin{aligned}
			|y(\xi;u_0)-y(\xi;\bar{u})|\leq \max\bigg\{	\xi^{-1}\bigg(\xi+M\frac{a\epsilon}{(1+a)^2}\sigma^2\big(y(\xi;u_0)\big);\bar{u}\bigg)-\xi^{-1}(\xi;\bar{u}), \\
			\xi^{-1}(\xi;\bar{u})-\xi^{-1}\bigg(\xi-M\frac{a\epsilon}{(1+a)^2}\sigma^2\big(y(\xi;u_0)\big);\bar{u}\bigg)\bigg\}.
		\end{aligned}
	\end{equation}
	
1.3. We now consider the case $\xi \geq \frac{1}{a}$.
	It follows from \eqref{sigma} and \eqref{inequality} that
	\begin{equation*}
		\begin{aligned}
			|y(\xi;u_0)-y(\xi;\bar{u})|\leq \max\bigg\{\xi^{-1}\bigg(\xi+M\frac{a\epsilon}{(1+a)^2};\bar{u}\bigg)-\xi^{-1}(\xi;\bar{u}), \\
			\xi^{-1}(\xi;\bar{u})-\xi^{-1}\bigg(\xi-M\frac{a\epsilon}{(1+a)^2}; \bar{u}\bigg)\bigg\}.
		\end{aligned}
	\end{equation*}
		A direct calculation gives
	\begin{align*}
		{\xi^{-1}}'(\xi;\bar{u})=\frac{1}{\xi'(y(\xi;\bar{u});\bar{u})}=\frac{1}{1-e^{-ay(\xi;\bar{u})}}\leq 1+\frac{1}{ay(\xi;\bar{u})}\lesssim 1+\frac{1}{a\xi}\lesssim 1,
	\end{align*}
	for all $\xi\geq \frac{1}{2a}$, where we have utilized the following inequality
	\begin{align}\label{inequality2}
		1-e^{-x}\geq \frac{x}{1+x}, \quad \forall\, x\geq 0,
	\end{align}
	as well as the conclusion of Lemma \ref{baru}, which states that $y(\xi;\bar{u}) \sim \xi$ whenever $\xi\geq \frac{1}{2a}$.

	Thus by the mean value theorem, we have
	\begin{align}\label{initial3}
		|y(\xi;u_0)-y(\xi;\bar{u})|\lesssim \frac{a\epsilon}{(1+a)^2} \quad \text{for} \quad \xi \geq \frac{1}{a},
	\end{align}	
which yields that
	\begin{align}\label{initial4}
			|u_{0}(y(\xi;u_0))-\bar{u}(y(\xi;\bar{u}))|&\leq 	|u_{0}(y(\xi;u_0))-\bar{u}(y(\xi;u_0))| +|\bar{u}(y(\xi;u_0))-\bar{u}(y(\xi;\bar{u}))| \nonumber \\
			&\leq \|u_0-\bar{u}\|_{L_{y}^{\infty}}
			+ \|\bar{u}'\|_{L_{y}^{\infty}}|y(\xi;u_0)-y(\xi;\bar{u})| \nonumber\\
			&\lesssim \|u_0-\bar{u}\|_{L_{y}^{2}}^{\frac{1}{2}}\|\partial_y(u_0-\bar{u})\|_{L_{y}^{2}}^{\frac{1}{2}}+\frac{a^2\epsilon}{(1+a)^2}\lesssim \epsilon,\quad \xi \geq \frac{1}{a}.
	\end{align}
		Noting \eqref{inequality2}, we have
	\begin{align*}
		\bar{u}(y(\xi;\bar{u}))=1-e^{-ay(\xi;\bar{u})}\geq \frac{ay(\xi;\bar{u})}{1+ay(\xi;\bar{u})}\gtrsim 1 \quad \text{for} \quad \xi\geq \frac{1}{a},
	\end{align*}
	which, together with \eqref{initial4}, yields that
	\begin{align}\label{initial5}
		|u_{0}(y(\xi;u_0))-\bar{u}(y(\xi;\bar{u}))|\lesssim \epsilon\bar{u}(y(\xi;\bar{u})) \quad \text{for} \quad \xi \geq \frac{1}{a}.
	\end{align}
	
\smallskip

1.4. We now consider the case $\xi\leq \frac{1}{a}$.
	It follows from \eqref{sigma} and \eqref{inequality} that
	\begin{equation}\label{initial6.5}
		\begin{aligned}
			\big|y(\xi;u_0)-y(\xi;\bar{u})\big|
			\leq \max\bigg\{\xi^{-1}\big(\xi+ \frac{Ma\epsilon}{(1+a)^2} y^2(\xi;u_0);\bar{u}\big)-\xi^{-1}(\xi;\bar{u}),\\
			\xi^{-1}(\xi;\bar{u})-\xi^{-1}\big(\xi- \frac{Ma\epsilon}{(1+a)^2} y^2(\xi;u_0);\bar{u}\big) \bigg\}.
		\end{aligned}
	\end{equation}
	Now we proceed to estimate $y^2(\xi;u_0)$.
	
	It follows from \eqref{initial2} that there exists a constant $M_1$, independent of $a$, such that
	\begin{align}\label{initial7}
		\xi\big((1-M_1\epsilon)y;\bar{u}\big)\leq \xi(y;u_0) \leq \xi\big((1+M_1\epsilon)y;\bar{u}\big).
	\end{align}
	We only verify the left-hand side of \eqref{initial7}, as the right-hand side follows by a similar argument.
	By the mean value theorem,
	\begin{align*}
		\xi(y;\bar{u})-\xi\big((1-M_1\epsilon)y;\bar{u}\big)
		= \xi'\big(y-\mu M_1 \epsilon y;\bar{u}\big) M_1\epsilon y,\quad \text{for some} \quad \mu \in [0,1].
	\end{align*}
	In view of \eqref{initial2}, it suffices to show that
	\begin{align*}
		\xi'\big(y-\mu M_1 \epsilon y;\bar{u}\big) M_1\epsilon y
		\geq \frac{Ma\epsilon}{(1+a)^2}\sigma^2(y),\quad \forall\, y\geq 0.
	\end{align*}

For the case $y\geq \frac{2}{a}$, since $\epsilon\ll 1$, it is enough to guarantee
	\begin{align*}
		\xi'\big(\tfrac{1}{a};\bar{u}\big)M_1\epsilon y\geq \frac{Ma\epsilon}{(1+a)^2}\sigma^2(y),
		\quad \text{i.e.,} \quad M_1\geq M_2\frac{\sigma^2(y)}{y}.
	\end{align*}
Due to $\frac{\sigma^2(y)}{y}\leq 1$
it is clear that  such a constant $M_1$ can be suitably chosen.
	
Next, we consider the case $y\leq \frac{2}{a}$. Recall that $\xi'(y;\bar{u})\gtrsim ay$ for $y\leq \frac{2}{a}$. For sufficiently small $\epsilon$, we have
	\begin{align*}
		\xi'\big(y-\mu M_1 \epsilon y; \bar{u}\big) M_1\epsilon y
		\geq \epsilon M_1M_3ay^2
		\geq \frac{Ma\epsilon}{(1+a)^2}\sigma^2(y)
		\quad \text{for} \quad M_1\geq M_4.
	\end{align*}
	Combining the above two cases, we can choose a suitable constant $M_1$, independent of $a$, such that inequality \eqref{initial7} holds for all $y\geq 0$.
	
	Thus we have from \eqref{initial7} that
	\begin{align}\label{initial7.5}
		\frac{\xi^{-1}(\xi;\bar{u})}{1+M_1\epsilon} \leq y(\xi;u_0)\leq \frac{\xi^{-1}(\xi;\bar{u})}{1-M_1\epsilon}.
	\end{align}
	Recall $\xi^{-1}(\xi;\bar{u})=y(\xi;\bar{u})$. Combining \eqref{initial7.5} with Lemma \ref{baru}, we obtain
	\begin{align}\label{initial8}
		y^2(\xi;u_0)\sim y^2(\xi;\bar{u}), \quad \forall \xi\leq \frac{1}{a}.
	\end{align}
	
We have from Lemma \ref{baru} that
	\begin{align*}
		\big(\xi^{-1}\big)'(\xi;\bar{u})
		=\frac{1}{\xi'\big(y(\xi;\bar{u});\bar{u}\big)}
		=\frac{1}{\bar{u}\big(y(\xi;\bar{u})\big)}
		\sim \frac{1}{a\,y(\xi;\bar{u})},
		\quad \forall \xi \leq \frac{2}{a},
	\end{align*}
	which, together with \eqref{initial6.5}, and \eqref{initial8}, yields that
	\begin{align}\label{initial9}
		\big|y(\xi;u_0)-y(\xi;\bar{u})\big|
		\lesssim \frac{M\epsilon}{(1+a)^2}\, y(\xi;\bar{u}),
		\quad \forall \xi \leq \frac{1}{a}.
	\end{align}
	
Using \eqref{initial9}, we proceed with a similar argument as for \eqref{initial4} to derive
	\begin{align}\label{initial9.5}
			\big|u_{0}\big(y(\xi;u_0)\big)-\bar{u}\big(y(\xi;\bar{u})\big)\big|
			\leq& \big|u_{0}\big(y(\xi;u_0)\big)-\bar{u}\big(y(\xi;u_0)\big)\big| + \big|\bar{u}\big(y(\xi;u_0)\big)-\bar{u}\big(y(\xi;\bar{u})\big)\big| \nonumber \\
			\leq& \big\|\partial_y(u_0-\bar{u})\big\|_{L_{y}^{\infty}}\,y(\xi;u_0) + \|\bar{u}'\|_{L_y^{\infty}} |y(\xi;u_0)-y(\xi;\bar{u}) | \nonumber\\
			\lesssim& \epsilon a y(\xi;\bar{u})
			\lesssim \epsilon \bar{u}\big(y(\xi;\bar{u})\big), \quad \forall \xi \leq \frac{1}{a},
	\end{align}
where we have used \eqref{equiva} in the last inequality.
	
\smallskip
	
1.5. Combining \eqref{initial5} and \eqref{initial9.5}, one obtains
	\begin{align}
		\big|u_{0}^2\big(y(\xi;u_0)\big)-\bar{u}^2\big(y(\xi;\bar{u})\big)\big|
		\lesssim \epsilon\, \bar{u}^2\big(y(\xi;\bar{u})\big),
		\quad \forall \xi\geq 0.
	\end{align}
	
2. Using the standard parabolic maximum principle and following the argument of Lemma 2.2 in \cite{Lei}, we can obtain
	\begin{align}\label{initial9.9}
		|\phi(X,\xi)|\lesssim \bar{u}^2\big(y(\xi;\bar{u})\big).
	\end{align}
	Noting \eqref{initial9.9} and $\epsilon\ll 1$, one gets
	\begin{align*}
		u^2\big(X,y(X,\xi;u)\big) \sim \bar{u}^2\big(y(\xi;\bar{u})\big),
	\end{align*}
	which implies that
	\begin{align}
		|\rho(X,\xi)|\lesssim \frac{|\phi(X,\xi)|}{\bar{u}}
		\lesssim \bar{u}\big(y(\xi;\bar{u})\big).
	\end{align}
	Therefore the proof of Lemma \ref{u0simbaru} is completed.
\end{proof}

	\begin{lem}[\cite{Hardy}]\label{lem1}
	Let $f$ be an absolutely continuous function on $[0,\infty)$,  and $s\neq\frac{1}{2}$. Then it holds that
		\begin{align}
		\Big\| \frac{f}{\xi^s}\Big\|_{L^2}\leq
		\begin{cases}
		\displaystyle	\frac{2}{2s-1} 	\Big\| \frac{f'}{\xi^{s-1}}\Big\|_{L^2},  & if\ s>\frac{1}{2} \ and\  f(0)=0, \\[10pt]
			\displaystyle
			\frac{2}{1-2s}\Big\| \frac{f'}{\xi^{s-1}}\Big\|_{L^2}, & if\ s<\frac{1}{2} \ and\  f(+\infty)=0.
		\end{cases} 
		\end{align}
	\end{lem} 
We now establish the following sharp Hardy-type estimate, which plays a crucial role in the subsequent proof.
	\begin{lem}\label{Lower base}Given any fixed positive constant $\beta\neq1$.
	  Assume $f$ is absolutely continuous in every bounded interval of  $\mathbb{R}$ and $f(0)=0$, then   \begin{align}\label{hardy-type}
	  	\Big\| \frac{\bar{u}_y f\chi_{\frac{\delta}{a}}}{\bar{u}^{\beta}}\Big\|_{L_{\xi}^2}
	  	\leq& \frac{1}{C_2(\delta)^{2}}\frac{2}{|\beta-1|}   \left(\Big\|  \frac{\bar{u}_y f\chi_{\frac{\delta}{a}}}{\bar{u}^{\beta-1}}\Big\|_{L_{\xi}^2}    
	  	+        \Big\|  \frac{ f_\xi\chi_{\frac{\delta}{a}}}{\bar{u}^{\beta-2}}\Big\|_{L_{\xi}^2}   +\frac{C}{\delta}  \Big\| \frac{\bar{u}_y f}{\bar{u}^{\beta-2}}   \Big\|_{L_{\xi\geq \frac{\delta}{a}}^2}                        \right)   .
	  \end{align}where  $\chi_{\frac{\delta}{a}}$ is the cut-off function defined in \eqref{cut},  $C>0$ is a constant depending only on $\chi'$ and the constant $C_2(\delta)$ satisfies \begin{align}\label{sharp constant}
	  \lim_{\delta\to 0}C_2(\delta)=\sqrt{2}.
	  \end{align} 
	 Here we remark that \eqref{sharp constant} is essential for estimates \eqref{F11} and \eqref{weighted F11}.
	\end{lem}

\begin{proof}
Using Remark \ref{H,h} and Lemma \ref{lem1}, one has
	\begin{align*}
		\bigg\| \frac{\bar{u}_y f\chi_{\frac{\delta}{a}}}{\bar{u}^{\beta}}\bigg\|_{L_{\xi}^2}
		\leq& \frac{1}{h(2\delta)^{\beta}} \bigg\| \frac{\bar{u}_y f\chi_{\frac{\delta}{a}}}{(a\xi)^{\frac{\beta}{2}}}\bigg\|_{L_{\xi}^2}\\
		\leq& \frac{1}{h(2\delta)^{\beta}}\frac{2}{|\beta-1|}  \bigg\| \left(\frac{\bar{u}_{yy}}{\bar{u}}f\chi_{\frac{\delta}{a}}+\bar{u}_yf_\xi \chi_{\frac{\delta}{a}}+\frac{a}{\delta}\bar{u}_yf \chi_{\frac{\delta}{a}}' \right) \frac{\xi}{(a\xi)^{\frac{\beta}{2}}}\bigg\|_{L_{\xi}^2}\\	
		=& \frac{1}{h(2\delta)^{\beta}}\frac{2}{|\beta-1|}    \bigg\| \left(\frac{-\bar{u}_{y}}{\bar{u}}f\chi_{\frac{\delta}{a}}+\bar{u}_y\frac{f_\xi}{a} \chi_{\frac{\delta}{a}}+\frac{1}{\delta}\bar{u}_yf \chi_{\frac{\delta}{a}}' \right) \frac{1}{(a\xi)^{\frac{\beta}{2}-1}}\bigg\|_{L_{\xi}^2}\\
		\leq& \frac{1}{C_2(\delta)^{2}}\frac{2}{|\beta-1|}    \bigg\| \left(\frac{-\bar{u}_{y}}{\bar{u}}f\chi_{\frac{\delta}{a}}+\bar{u}_y\frac{f_\xi}{a} \chi_{\frac{\delta}{a}}+\frac{1}{\delta}\bar{u}_yf \chi_{\frac{\delta}{a}}' \right) \frac{1}{\bar{u}^{\beta-2}}\bigg\|_{L_{\xi}^2}\\
		\leq& \frac{1}{C_2(\delta)^{2}}\frac{2}{|\beta-1|} \Bigg( \bigg\|  \frac{\bar{u}_y f\chi_{\frac{\delta}{a}}}{\bar{u}^{\beta-1}}\bigg\|_{L_{\xi}^2}
		+ \bigg\|  \frac{ f_\xi\chi_{\frac{\delta}{a}}}{\bar{u}^{\beta-2}}\bigg\|_{L_{\xi}^2}
		+\frac{C}{\delta}  \bigg\| \frac{\bar{u}_y f}{\bar{u}^{\beta-2}}   \bigg\|_{L_{\xi\geq \frac{\delta}{a}}^2} \Bigg),
	\end{align*}
where
	\begin{equation}
		C_2(\delta)^2=
		\begin{cases}
			\displaystyle h(2\delta)^2, & \beta\le 2,\\
		\displaystyle 	\frac{h(2\delta)^\beta}{H(2\delta)^{\beta-2}}, & \beta\ge 2.   
		\end{cases} 
	\end{equation} 
	Therefore the proof of Lemma \ref{Lower base} is completed.
\end{proof}

 \subsection{Proof of Theorem \ref{Main}}\label{subsection2.2}
	\begin{lem}\label{lem2}
		Let $\phi$ solve  \eqref{eqphi}. Under the assumptions of Theorem \ref{Main}, there exists a sufficiently small constant $\epsilon>0$ such that
		\begin{align}\label{F00}
		\sup_{X \geq0}
		\int\phi^2  + \int\int u \phi_{\xi}^2 
			+\int\int A \phi^2  \lesssim \|\phi_0\|^2_{L^2_\xi} \lesssim \|\phi_0\|_{\mathbf{Y}_{0,0}}^2.
		\end{align}
	\end{lem}
	
\begin{proof}
Multiplying \eqref{eqphi} by $\phi$ and noting $u=\bar{u}+\epsilon \rho$, we obtain the energy identity
	\begin{align}\label{F01}
		\frac{1}{2} \frac{d}{d X} \int \phi^2  +\int u \phi_{\xi}^2-\frac{1}{2}\int \bar{u}_{\xi\xi}\phi^2
		+\int A \phi^2 =- \epsilon \int \rho_\xi\phi\phi_{\xi}.
	\end{align}
	For the  term on RHS, noting $\phi=(2\bar{u}+\epsilon\rho)\rho$, we have
\begin{align}\label{rhoxi}
	\rho_{\xi}
	=\frac{\phi_{\xi}}{2u} - \frac{\bar{u}_{\xi}}{u}\rho
	=\frac{\phi_{\xi}}{2u} - \frac{\bar{u}_{y}}{\bar{u}u}\rho.
\end{align}
which, together with Lemma \ref{u0simbaru}, yields that
\begin{align}\label{F02}
	\bigg|\int  \rho_\xi\phi_{\xi}\phi\,\mathrm{d}\xi\bigg|
	\lesssim \int \bar{u}\,\phi_{\xi}^2\mathrm{d}\xi
	+ \int  \frac{\bar{u}_{y}\,|\phi\phi_\xi|}{\bar{u}}\mathrm{d}\xi.
\end{align}
Noting $\bar{u}\sim u$, $\bar{u}_{y}^2\leq |\bar{u}_{yy}|$ and \eqref{hardy-type}, one has
	\begin{align}\label{F03}
			\int  \frac{\bar{u}_{y}|\phi\phi_\xi|}{\bar{u}}
			\leq& \Big\|\sqrt{u} \phi_{\xi} \Big\|_{L_{\xi}^2} \Big\| \frac{\bar{u}_{y}\phi}{\bar{u}^{\frac{3}{2}}}\Big\|_{L_{\xi}^2}  
			\leq \Big\|\sqrt{u} \phi_{\xi} \Big\|_{L_{\xi}^2}  \left[ \Big\| \frac{\bar{u}_{y}\phi  \chi_\frac{1}{a} }{\bar{u}^{\frac{3}{2}}}\Big\|_{L_{\xi}^2}+\Big\| \frac{\bar{u}_{y}\phi}{\bar{u}^{\frac{3}{2}}}\Big\|_{L_{\xi\geq\frac{1}{a}}^2}  \right] \nonumber\\
			\lesssim& \Big\| \sqrt{u}\phi_{\xi} \Big\|_{L_{\xi}^2}  \left[ \Big\|  \frac{\bar{u}_y \phi\chi_{\frac{1}{a}}}{\bar{u}^{\frac{1}{2}}}\Big\|_{L_{\xi}^2}    
			+ \Big\|  \sqrt{\bar{u}}\phi_\xi\chi_{\frac{1}{a}}\Big\|_{L_{\xi}^2}   
			+ \Big\| \sqrt{\bar{u}}\bar{u}_y \phi   \Big\|_{L_{\xi\geq \frac{\delta}{a}}^2}                 
			+\Big\| \frac{\bar{u}_{y}\phi}{\bar{u}^{\frac{3}{2}}}\Big\|_{L_{\xi\geq\frac{1}{a}}^2}   \right] \nonumber\\
			\lesssim& \Big\| \sqrt{u}\phi_{\xi} \Big\|_{L_{\xi}^2}  \left[ \Big\|\sqrt{u} \phi_\xi\Big\|_{L_{\xi}^2}+ \left(\int A\phi^2 \right)^{\frac{1}{2}}  \right] \nonumber\\
			\lesssim& \Big\|\sqrt{u} \phi_\xi\Big\|_{L_{\xi}^2}^2+ \int A \phi^2.
	\end{align}
Noting
\begin{align}\label{F03.5}
	\bar{u}_{\xi\xi}
	=\left(\frac{\bar{u}_y}{\bar{u}}\right)_\xi
	=\frac{\bar{u}_{yy}}{\bar{u}^2}-\frac{\bar{u}_{y}^2}{\bar{u}^3}\leq0,
\end{align}
then substituting \eqref{F02}-\eqref{F03} into \eqref{F01}, and using \eqref{F03.5}, one has 
\begin{align}\label{F04}
	\frac{d}{d X} \int \phi^2
	+ \int u\phi_{\xi}^2
	+ \int A \phi^2
	\le 0,
\end{align}where we have used $0<\epsilon\ll 1$.
	
	Integrating \eqref{F04}  over $[0,X]$,
	we derive the desired energy estimate \eqref{F00}. Therefore the proof of Lemma \ref{lem2} is completed.
\end{proof}

	\begin{lem}\label{lem3}
		Let $\phi$ solve \eqref{eqphi} and $\lambda\in [1,\frac32)$. Under the assumptions in Theorem \ref{Main}, it holds that
		\begin{align}\label{F05}
		\| \phi\|_{\mathbf{Y}_0}^2\equiv \left\|
		\frac{\phi_\xi}{u^{\lambda-\frac{1}{2}}}
		\right\|_{L^2_{X, \xi }}^2+ \ \sup_{X \ge 0}
		\left\|
		\frac{\phi}{u^{\lambda}}
		\right\|_{L^2_{\xi}}^2 \lesssim  \| \phi_0\|_{\mathbf{Y}_{0,0}}^{2},
		\end{align}
		provided $0<\epsilon\ll 1$.
	\end{lem}
\begin{proof}
Denote $k=2\lambda-1$, it is clear that $1\le k<2$. Multiplying \eqref{eqphi} by  $\phi/\bar{u}^{k+1} $,  one gets
	\[
	\frac12 \partial_X \left(\phi^2\right)\frac{1}{\bar{u}^{k+1}}
	-\frac{a}{2}\partial_\xi \left(\phi^2\right)\frac{1}{\bar{u}^{k+1}}
	-\left(\frac1{\bar{u}^k}+\epsilon\frac{\rho}{\bar{u}^{k+1}}\right)\frac12\partial_{\xi\xi}(\phi^2)
	+\frac{\bar{u}+\epsilon\rho}{\bar{u}^{k+1}}\phi_\xi^2
	+A\frac{\phi^2}{\bar{u}^{k+1}}=0.
	\]
Noting the definition of $\bar{u}$, one easily obtains 
	\begin{align}\label{F06}
	-a\bar{u}_\xi=\bar{u}_{yy}/\bar{u},
	\end{align}
then one has from integration by parts that
	\begin{equation}\label{F07}
		\begin{aligned}
			&	\frac12 \frac{d}{d X}\int\frac{\phi^2}{\bar{u}^{k+1}}
			+\int\frac{\bar{u}+\epsilon\rho}{\bar{u}^{k+1}}\phi_\xi^2\\
			&=\,k\int \phi\phi_\xi\frac{\bar{u}_y}{\bar{u}^{k+2}}
			-\epsilon\int \left(\frac{\rho}{\bar{u}^{k+1}}\right)_\xi \phi\phi_\xi 
			+\int\left(\frac1{1+\frac{\epsilon\rho}{2\bar{u}}}-\frac{k+1}{2}\right)\frac{\bar{u}_{yy}\phi^2}{\bar{u}^{k+3}}.
		\end{aligned}
	\end{equation}

For the 1st term  on RHS of \eqref{F07},
it holds that
	\begin{align}\label{F08}
		\left|\int \phi\phi_\xi \frac{\bar{u}_y}{\bar{u}^{k+2}}\right|
		\le \left\|\frac{\phi_\xi}{\bar{u}^{\frac{k}{2}}}\right\|_{L_\xi^2}
		\left\|\frac{\phi\bar{u}_y\chi_\frac{\delta}{a}}{\bar{u}^{\frac k2+2}}\right\|_{L_\xi^2}
		+\delta\left\|\frac{\phi_\xi}{\bar{u}^{\frac{k}{2}}}\right\|_{L_\xi^2}^2
		+C(\frac{1}{\delta})\int A\phi^2,
	\end{align}
	where we have used $\bar{u}_y^2\lesssim|\bar{u}_{yy}|$ and  $\bar{u}\sim_{\delta} 1$ for $\xi\geq \frac{\delta}{a}$ from Lemma \ref{baru}.
	
	Using Lemma \ref{Lower base}, one derives 
\begin{align}\label{F09}
	\left\|\frac{\phi\bar{u}_y\chi_{\frac{\delta}{a}}}{\bar{u}^{\frac k2+2}}\right\|_{L_\xi^2}
	\leq \frac{1}{C_2(\delta)^{2}}\frac{4}{k+2} \Bigg(
	\left\|  \frac{\bar{u}_y \phi\chi_{\frac{\delta}{a}}}{\bar{u}^{\frac{k}{2}+1}}\right\|_{L_{\xi}^2}
	+ \left\|  \frac{ \phi_\xi\chi_{\frac{\delta}{a}}}{\bar{u}^{\frac{k}{2}}}\right\|_{L_{\xi}^2}
	+\frac{C}{\delta} \left\| \frac{\bar{u}_y \phi}{\bar{u}^{\frac{k}{2}}} \right\|_{L_{\xi\geq \frac{\delta}{a}}^2}
	\Bigg).
\end{align}
We again invoke Lemma \ref{Lower base} to get
\begin{align}\label{F09.5}
	\left\|  \frac{\bar{u}_y \phi\chi_{\frac{\delta}{a}}}{\bar{u}^{\frac{k}{2}+1}}\right\|_{L_{\xi}^2}
	&\lesssim \left\|  \frac{\bar{u}_y \phi\chi_{\frac{\delta}{a}}}{\bar{u}^{\frac{k}{2}}}\right\|_{L_{\xi}^2}
	+\left\|  \frac{ \phi_\xi\chi_{\frac{\delta}{a}}}{\bar{u}^{\frac{k}{2}-1}}\right\|_{L_{\xi}^2}
	+\frac{C}{\delta}\left\| \frac{\bar{u}_y \phi}{\bar{u}^{\frac{k}{2}-1}} \right\|_{L_{\xi\geq \frac{\delta}{a}}^2} \nonumber\\
		&\lesssim  \left\|  \frac{\bar{u}_y \phi\chi_{\frac{\delta}{a}}}{\bar{u}^{\frac{k}{2}}}\right\|_{L_{\xi}^2}
		+\sqrt{\delta} \left\|  \frac{ \phi_\xi\chi_{\frac{\delta}{a}}}{\bar{u}^{\frac{k}{2}}}\right\|_{L_{\xi}^2}
		+\frac{C}{\delta}\left\| \frac{\bar{u}_y \phi}{\bar{u}^{\frac{k}{2}-1}} \right\|_{L_{\xi\geq \frac{\delta}{a}}^2},
\end{align}
where we have used the fact $C_2(\delta)\sim 1$ for $\delta\ll 1$ and $\bar{u}\sim \sqrt{a\xi}$ for $\xi\leq \frac{1}{a}$.

Noting $1\leq k<2$, then we have from \eqref{F08}-\eqref{F09.5} that
\begin{align}\label{F11}
			\int \left|\phi\phi_\xi \frac{\bar{u}_y}{\bar{u}^{k+2}}  \right|
		\leq&
		\frac{1}{C_2(\delta)^{2}}\frac{4}{k+2} \left\|\frac{\phi_\xi}{\bar{u}^{\frac{k}{2}}}\right\|_{L_\xi^2}\Bigg(
	\big(C\sqrt{\delta}+1\big) \left\|  \frac{ \phi_\xi\chi_{\frac{\delta}{a}}}{\bar{u}^{\frac{k}{2}}}\right\|_{L_{\xi}^2}
		+C(\frac{1}{\delta})\left\| \frac{\bar{u}_y \phi}{\bar{u}^{\frac{k}{2}}} \right\|_{L_{\xi}^2}	\Bigg) \nonumber\\
		&
		+\delta\left\|\frac{\phi_\xi}{\bar{u}^{\frac{k}{2}}}\right\|_{L_\xi^2}^2
		+C(\frac{1}{\delta})\int A\phi^2
	\nonumber	\\
		\leq&\left(\frac{4}{(k+2)C_2(\delta)^2}+C\sqrt{\delta} \right)\left\|\frac{\phi_\xi}{\bar{u}^{k/2}}\right\|_{L_\xi^2}^2
		+C(\frac{1}{\delta})\int A\phi^2.
\end{align}

\smallskip

For the 3rd term on RHS of \eqref{F07},
it is clear that
\begin{equation}\label{F12}
	\begin{aligned}
		\left\|\frac{\phi\,|\bar{u}_{yy}|^{\frac{1}{2}}}{\bar{u}^{\frac{k+1}{2}+1}}\right\|_{L_\xi^2}^2
		\lesssim
		\left\|\frac{\phi\,|\bar{u}_{yy}|^{\frac{1}{2}} \chi_{\frac{\delta}{a}} }{\bar{u}^{\frac{k+1}{2}+1}}\right\|_{L_\xi^2}^2
		+
		C(\frac{1}{\delta})\int A \phi^2.
	\end{aligned}
\end{equation}
For $1<k<2$, using similar arguments as in Lemma \ref{Lower base}
twice, we obtain
\begin{align}\label{F13}
			\left\|\frac{\phi\,|\bar{u}_{yy}|^{\frac{1}{2}} \chi_{\frac{\delta}{a}} }{\bar{u}^{\frac{k+1}{2}+1}}\right\|_{L_\xi^2}^2
		&\lesssim
		\left\|  \frac{|\bar{u}_{yy}|^{\frac{1}{2}} \phi\,\chi_{\frac{\delta}{a}}}{\bar{u}^{\frac{k+1}{2}}}\right\|_{L_{\xi}^2}^2
		+\left\|  \frac{\phi_\xi\,\chi_{\frac{\delta}{a}}}{\bar{u}^{\frac{k+1}{2}-1}}\right\|_{L_{\xi}^2}^2
		+\frac{C}{\delta^2}\left\| \frac{|\bar{u}_{yy}|^{\frac{1}{2}} \phi}{\bar{u}^{\frac{k+1}{2}-1}} \right\|_{L_{\xi\ge \frac{\delta}{a}}^2}^2 \nonumber\\
		&\lesssim 	\left\|  \frac{|\bar{u}_{yy}|^{\frac{1}{2}} \phi\,\chi_{\frac{\delta}{a}}}{\bar{u}^{\frac{k+1}{2}-1}}\right\|_{L_{\xi}^2}^2+\left\|  \frac{\phi_\xi\,\chi_{\frac{\delta}{a}}}{\bar{u}^{\frac{k+1}{2}-1}}\right\|_{L_{\xi}^2}^2
		+\frac{C}{\delta^2}\left\| \frac{|\bar{u}_{yy}|^{\frac{1}{2}} \phi}{\bar{u}^{\frac{k+1}{2}-1}} \right\|_{L_{\xi\ge \frac{\delta}{a}}^2}^2  \nonumber\\
		&\lesssim \sqrt{\delta}\left\|\frac{\phi_\xi}{\bar{u}^{k/2}}\right\|_{L_\xi^2}^2
		+C(\frac{1}{\delta})\int A\phi^2,
\end{align}
where the last estimate in \eqref{F13} also holds for $k=1$ by a similar argument.
\smallskip

For the second term on RHS of \eqref{F07},
it is clear that
	\begin{align*}
		\epsilon\int\left(\frac{\rho}{\bar{u}^{k+1}}\right)_\xi\phi\phi_\xi
		=\epsilon\int\frac{\rho_\xi}{\bar{u}^{k+1}}\phi\phi_\xi
		-\epsilon(k+1)\int\frac{\rho\bar{u}_y}{\bar{u}^{k+3}}\phi\phi_\xi.
	\end{align*}
	Using \eqref{rhoxi} and Lemma \ref{u0simbaru}, one has
	\begin{align}\label{F14}
		\left|\epsilon\int\left(\frac{\rho}{\bar{u}^{k+1}}\right)_\xi\phi\phi_\xi\right|
		\lesssim\epsilon\left\|\frac{\phi_\xi}{\bar{u}^{k/2}}\right\|_{L_\xi^2}^2
		+\epsilon \int\frac{\bar{u}_y|\phi\phi_\xi|}{\bar{u}^{k+2}},
	\end{align}
	which has been estimated in \eqref{F11}.
	
	Recall the crucial estimate \eqref{sharp constant}, i.e., that $	\lim_{\delta\to 0}C_2(\delta)=\sqrt{2}$.
	 Substituting \eqref{F08}-\eqref{F14} into \eqref{F07}, we obtain 
	\[
	\frac12\frac{d}{d X} \int\frac{\phi^2}{\bar{u}^{k+1}}+C(k)\int\frac{\phi_\xi^2}{\bar{u}^k}
	\lesssim \int A\phi^2,
	\]where 
		$C(k)=1-\frac{2k}{k+2}-o_{\delta,\epsilon}(1).$
		
	Integrating  over $[0,X]$, and noting $k=2\lambda-1$, one has
	\[
	\|\phi\|_{\mathbf{Y}_0}^2\lesssim\|\phi_0\|_{\mathbf{Y}_{0,0}}^2+\int_0^\infty\!\!\int A\phi^2,
	\]
	which, together with Lemma \ref{lem2}, concludes  \eqref{F05}.
	Therefore  the proof of Lemma \ref{lem3} is completed.
\end{proof}
	\begin{lem}\label{lem4}
		Let $\phi$ solve \eqref{eqphi}. There exists $\epsilon>0$  sufficiently small such that under the assumptions in Theorem \ref{Main}, it holds that
		\begin{align}\label{F15}
			\| \phi\|_{\mathbf{Y}_1}^{2}\equiv \left\|
			\frac{\phi_X}{\sqrt{u}}
			\right\|_{L^2_{X , \xi }}^2 + \ \sup_{X \ge 0}
			\left\|
			\phi_\xi
			\right\|_{L^2_{\xi }}^2 \lesssim a^2\| \phi_0 \|_{\mathbf{Y}_{0,0}}^{2}+
			\| \phi_0 \|_{\mathbf{Y}_{1,0}}^{2}.
		\end{align}
	\end{lem}
	\begin{proof}
		Multiplying \eqref{eqphi} by $\dfrac{\phi_X}{u}$, one has
		\begin{align}\label{F16}
			\int \frac{\phi_{X}^2}{u}
			+\frac12 \frac{d}{d X}\int \phi_{\xi}^2
			=a\int \frac{\phi_{\xi}\phi_{X}}{u}
			-\int A \phi \frac{\phi_{X}}{u}.
		\end{align}
			It is clear that
		\begin{align}\label{F17}
			a\left|\int \frac{\phi_{\xi}\phi_{X}}{u}\right|
			\leq \frac12 \int \frac{\phi_X^2}{u}
			+\frac{a^2}{2}\int \frac{\phi_\xi^2}{u}.
		\end{align}
			For the last term on RHS of \eqref{F16}, using Lemmas \ref{u0simbaru} \& \ref{lem1},     one has
		\begin{align}\label{E103}
				\left|\int A \phi \frac{\phi_X}{u}\,\chi_{\frac{1}{a}} \right|
				&\lesssim
				a\left\| \frac{\phi_X}{\sqrt{u}} \right\|_{L^2_\xi}
				\left\| \frac{\bar{u}_{y}\phi\,\chi_{\frac{1}{a}}}{u^{\frac52}} \right\|_{L^2_\xi} \nonumber\\
				&\lesssim
				a\left\| \frac{\phi_X}{\sqrt{u}} \right\|_{L^2_\xi}
				\left(
				\left\| \frac{\bar{u}_{y} \phi\,\chi_{\frac{1}{a}}}{\bar{u}^{\frac32}} \right\|_{L^2_\xi}
				+\left\| \frac{\phi_\xi\,\chi_{\frac{1}{a}}}{\bar{u}^{\frac12}} \right\|_{L^2_\xi}
				+C\left\| \frac{\bar{u}_{y} \phi}{\bar{u}^{\frac12}} \right\|_{L^2_{\xi\ge 1/a}}
				\right)\nonumber \\
				&\lesssim a
				\left\| \frac{\phi_X}{\sqrt{u}} \right\|_{L^2_\xi}
				\left(
				\left\| \frac{\phi_\xi\,\chi_{\frac{1}{a}}}{\bar{u}^{\frac12}} \right\|_{L^2_\xi}
				+\left(\int A\phi^2\right)^{\frac12}
				\right)\nonumber \\
				&\lesssim
				\delta \left\| \frac{\phi_X}{\sqrt{u}} \right\|_{L^2_\xi}^2
				+C(\frac{1}{\delta})a^2\int \frac{\phi_\xi^2}{u}
				+C(\frac{1}{\delta})a^2\int A\phi^2,
		\end{align}
	which yields immediately that 
		\begin{align}\label{F18}
			\left|\int A \phi \frac{\phi_X}{u} \right|
			\lesssim
			\delta \int \frac{\phi_X^2}{u}
			+C(\frac{1}{\delta})a^2\int \frac{\phi_\xi^2}{u}
		+C(\frac{1}{\delta})a^2\int A\phi^2.
		\end{align}
			Taking $\delta>0$ suitably small, we have from \eqref{F16}-\eqref{F18} that \begin{align}\label{F19}
		\frac{d}{dX}\int \phi_{\xi}^2+\int \frac{\phi_X^2}{u}\lesssim a^2\int \frac{\phi_\xi^2}{u}
		+a^2\int A\phi^2.
	\end{align}
	Finally, integrating \eqref{F19} over $[0,X]$, and using Lemmas \ref{lem2} \&  \ref{lem3}, we conclude \eqref{F15}.
		Therefore the proof of Lemma \ref{lem4} is completed.
	\end{proof}
	\begin{proof}[\textbf{Proof of Theorem \ref{Main}}]
	Combining Lemmas \ref{lem2}--\ref{lem4}, we conclude \eqref{esimate in Y0}-\eqref{estimate in Y1}.
	\end{proof}

	\section{Decay Estimates}
		\subsection{Decay estimates in von Mises coordinates}\label{decay von mises}

\begin{lem}\label{RATE2}
	Under the hypotheses of Theorem \ref{Main2},   for all integers $0\le k\le m$, it holds that
\begin{equation}\label{rate3}
		\begin{aligned}
			&\frac{d}{d X}\int \phi^2(b_{m,a}+\xi)^k
			+\frac{ka}{2}\int\phi^2 (b_{m,a}+\xi)^{k-1}
			+\int  u\phi_{\xi}^2(b_{m,a}+\xi)^k
			+\int A \phi^2 (b_{m,a}+\xi)^k \le 0,
		\end{aligned}
	\end{equation} where the constant $b_{m,a}\gg 1$  depends only on $m$ and $a$.
\end{lem}
\begin{proof}
	Multiplying  \eqref{eqphi} by $\phi(b_{m,a}+\xi)^k$, one obtains
	\begin{align*}
		\frac{1}{2}\big(\phi^2(b_{m,a}+\xi)^k\big)_X
		-a\phi_{\xi}\phi(b_{m,a}+\xi)^k
		-u\phi_{\xi \xi}\phi(b_{m,a}+\xi)^k
		+A\phi^2(b_{m,a}+\xi)^k=0.
	\end{align*}
	Integrating  by parts, one obtains
\begin{align}\label{21.5}
			&\frac12\frac{d}{d X}\int \phi^2(b_{m,a}+\xi)^k
			+ \frac{ka}{2}\int (b_{m,a}+\xi)^{k-1}\phi^2
			+\int  u\phi_{\xi}^2 (b_{m,a}+\xi)^k
			-\frac12\int \bar{u}_{\xi\xi}\phi^2(b_{m,a}+\xi)^k\nonumber\\
			&+\int A \phi^2 (b_{m,a}+\xi)^k\nonumber\\
			=&\frac{k}{2}\int\bar{u}_{\xi}\phi^2 (b_{m,a}+\xi)^{k-1}
			-k\int u\phi\phi_\xi  (b_{m,a}+\xi)^{k-1}
			-\epsilon \int \phi \phi_\xi \rho_\xi (b_{m,a}+\xi)^k.
	\end{align}	Using $\bar{u}_\xi=\dfrac{\bar{u}_y}{\bar{u}}$, we have
	\begin{align}
		k\int\bar{u}_{\xi}\phi^2(b_{m,a}+\xi)^{k-1}\lesssim \frac{m}{ab_{m,a}}\int A \phi^2 (b_{m,a}+\xi)^{k}.
	\end{align}
	It is clear that
	\begin{align}
		\bigg|k	\int  u\phi\phi_\xi (b_{m,a}+\xi)^{k-1} \bigg| \leq \frac{ka}{4}\int (b_{m,a}+\xi)^{k-1} \phi^2+\frac{m}{ab_{m,a}}\int u \phi_{\xi}^2(b_{m,a}+\xi)^{k}.
	\end{align}
	For the last term on RHS of \eqref{21.5}, using Lemma \ref{u0simbaru} and \eqref{rhoxi}, one has
	\begin{equation}
		\begin{aligned}
			\epsilon \bigg|\int \phi \phi_{\xi }\rho_{\xi} (b_{m,a}+\xi)^k \bigg|
			\lesssim& \epsilon \int \frac{\bar{u}_y}{\bar{u}}|\phi \phi_\xi |(b_{m,a}+\xi)^k\chi_{\frac{1}{a}} 
			+\epsilon \int  u\phi_{\xi}^2 (b_{m,a}+\xi)^k \\
			&+\epsilon \int A \phi^2 (b_{m,a}+\xi)^k .
		\end{aligned}
	\end{equation}
	Using Lemma \ref{Lower base}, one derives 
\begin{align*}
			&\epsilon\int \frac{\bar{u}_y}{\bar{u}}|\phi \phi_\xi| (b_{m,a}+\xi)^k\chi_{\frac{1}{a}} \\
			\lesssim&
			\epsilon\Big\| \sqrt{u}\phi_\xi (b_{m,a}+\xi)^{\frac{k}{2}} \Big\|_{L_{\xi}^2}
			\Big\| \frac{ \bar{u}_y \phi (b_{m,a}+\xi)^{\frac{k}{2}} \chi_{\frac{1}{a}} }{\bar{u}^{\frac{3}{2}}} \Big\|_{L_{\xi}^2} \\
			\lesssim&
			\epsilon\Big\| \sqrt{u}\phi_\xi(b_{m,a}+\xi)^{\frac{k}{2}} \Big\|_{L_{\xi}^2}
			\Biggr\{         
			\Big\|  \frac{\bar{u}_y \phi (b_{m,a}+\xi)^{\frac{k}{2}}   \chi_{\frac{1}{a}} }{\bar{u}^{\frac{1}{2}}}\Big\|_{L_{\xi}^2}
			+\Big\|  \frac{ \phi_\xi(b_{m,a}+\xi)^{\frac{k}{2}}  \chi_{\frac{1}{a}}}{\bar{u}^{-\frac{1}{2}}}\Big\|_{L_{\xi}^2}        \\
			&+  m\Big\|  \frac{ \phi (b_{m,a}+\xi)^{\frac{k}{2}-1} \chi_{\frac{1}{a}}}{\bar{u}^{-\frac{1}{2}}}\Big\|_{L_{\xi}^2}
			+\Big\| \frac{\bar{u}_y \phi(b_{m,a}+\xi)^{\frac{k}{2}}}{\bar{u}^{-\frac{1}{2}}}   \Big\|_{L_{\xi\geq \frac{1}{a}}^2}
			\Biggl\} \\
			\lesssim&  \epsilon \Big\| \sqrt{u}\phi_\xi (b_{m,a}+\xi)^{\frac{k}{2}} \Big\|_{L_{\xi}^2}
			\Biggr\{
			\Big\| \sqrt{u}\phi_\xi (b_{m,a}+\xi)^{\frac{k}{2}} \Big\|_{L_{\xi}^2}
			+ \frac{m}{b_{m,a}}\Big\|  \phi(b_{m,a}+\xi)^{\frac{k}{2}}\chi_{\frac{1}{a}}\Big\|_{L_{\xi}^2} \\
			&+\left(\int A\phi^2 (b_{m,a}+\xi)^k \right)^{\frac{1}{2}} 
			\Biggl\} \\
			\lesssim& \epsilon \Big\| \sqrt{u}\phi_\xi (b_{m,a}+\xi)^{\frac{k}{2}} \Big\|_{L_{\xi}^2}
			\Biggl\{
			\Big\| \sqrt{u}\phi_\xi (b_{m,a}+\xi)^{\frac{k}{2}} \Big\|_{L_{\xi}^2}
			+ \frac{m}{a b_{m,a}}\left(\int A\phi^2(b_{m,a}+\xi)^{k} \right)^{\frac{1}{2}}\\
			&	+\left(\int A \phi^2 (b_{m,a}+\xi)^{k} \right)^{\frac{1}{2}}
			\Biggr\} \\
			\lesssim \epsilon	& \Big\| \sqrt{u}\phi_\xi (b_{m,a}+\xi)^{\frac{k}{2}} \Big\|_{L_{\xi}^2}^2
			+ \left(\epsilon + \epsilon \frac{m^2}{a^2b_{m,a}^2} \right)\int A\phi^2 (b_{m,a}+\xi)^k.
	\end{align*}

Combining all the above estimates, using \eqref{F03.5} and  $\epsilon>0$ sufficiently small, we obtain
	\begin{align}
		&\frac{1}{2}	\frac{d}{d X}\int \phi^2(b_{m,a}+\xi)^k
		+\frac{ka}{4}\int (b_{m,a}+\xi)^{k-1}\phi^2
		+\frac{3}{4}\int  u\phi_{\xi}^2(b_{m,a}+\xi)^k
		+\frac{3}{4} \int A \phi^2 (b_{m,a}+\xi)^k  \notag\\
		\lesssim&  \left( \frac{m}{ab_{m,a}}+\epsilon \frac{m^2}{a^2b_{m,a}^2}  \right) \left(\int A \phi^2 (b_{m,a}+\xi)^{k} +\int u \phi_{\xi}^2(b_{m,a}+\xi)^{k}\right).
		\end{align}
Finally, taking $b_{m,a}\gg \frac{m}{a}$, one concludes \eqref{rate3}.
Therefore the proof of Lemma \ref{RATE2} is completed.
\end{proof}

\begin{lem}\label{RATE3}
	Under the assumptions of Theorem \ref{Main2}, it holds that
	\begin{equation}\label{30}
		\begin{aligned}
				\sup_{X \geq0}
			\int X^{m}\phi^2
			+\int\int  X^{m} u\phi_{\xi}^2 
			+ \int \int  X^{m}A \phi^2 
			\leq & C_{m,a} \|(1+\xi)^{\frac{m}{2}}\phi_{0}\|_{\mathbf{Y}_{0,0}}^{2}.
		\end{aligned}
	\end{equation} 
\end{lem}
\begin{proof}
	Integrating \eqref{rate3} with $k=m$ over $[0,L]$, we obtain
	\begin{equation}\label{32}
		\begin{aligned}
			&	\int \phi^2(b_{m,a}+\xi)^m
			+\frac{ma}{2}\int_{0}^{L}\int\phi^2 (b_{m,a}+\xi)^{m-1}
			+ \int_{0}^{L} \int  u\phi_{\xi}^2(b_{m,a}+\xi)^m
			\\
			&+ \int_{0}^{L}   \int A \phi^2 (b_{m,a}+\xi)^m	\lesssim \|(b_{m,a}+\xi)^{\frac{m}{2}}\phi_{0}\|_{\mathbf{Y}_{0,0}}^{2}.
		\end{aligned}
	\end{equation}
	For any integer $0\le k \le m-1$, we multiply \eqref{rate3} by $X^{m-k}$ to get
	\begin{align}
			&\frac{d}{d X}\int  X^{m-k}\phi^2(b_{m,a}+\xi)^k
			+\frac{ka}{2}\int  X^{m-k}\phi^2(b_{m,a}+\xi)^{k-1}
			+\int  X^{m-k} u\phi_{\xi}^2(b_{m,a}+\xi)^k \nonumber \\
			&+\int  X^{m-k}A \phi^2 (b_{m,a}+\xi)^k \nonumber\\
			\leq& (m-k) \int X^{m-k-1}\phi^2(b_{m,a}+\xi)^k,
	\end{align}
	which yields immediately that
	\begin{align}\label{33}
			&\int  L^{m-k}\phi^2(b_{m,a}+\xi)^k
			+ \frac{ka}{2} \int_{0}^{L}\int  X^{m-k}\phi^2(b_{m,a}+\xi)^{k-1}
			+\int_{0}^{L}\int  X^{m-k} u\phi_{\xi}^2(b_{m,a}+\xi)^k \nonumber\\
			&+ \int_{0}^{L} \int  X^{m-k}A \phi^2 (b_{m,a}+\xi)^k  \nonumber\\
			\leq& \|(b_{m,a}+\xi)^{\frac{m}{2}}\phi_{0}\|_{\mathbf{Y}_{0,0}}^{2} +(m-k) \int_{0}^{L}\int X^{m-k-1}\phi^2(b_{m,a}+\xi)^k.
	\end{align}
	Thus we have for $1\leq k\leq m-1$ that
	\begin{equation}
		\begin{aligned}
		&	\int_{0}^{L}\int  X^{m-k}\phi^2(b_{m,a}+\xi)^{k-1}\\
			\leq& \frac{2m}{a}\bigg(\|(b_{m,a}+\xi)^{\frac{m}{2}}\phi_{0}\|_{\mathbf{Y}_{0,0}}^{2}
			+ \int_{0}^{L}\int X^{m-k-1}\phi^2(b_{m,a}+\xi)^k  \bigg),
		\end{aligned}
	\end{equation}
	which, together with \eqref{32} and iterating the argument for $k = m-1,m-2,\dots,1$, yields that
	\begin{equation}\label{34}
		\begin{aligned}
			\int_{0}^{L}\int  X^{m-k}\phi^2(b_{m,a}+\xi)^{k-1}
			\leq C_{m,a} \|(b_{m,a}+\xi)^{\frac{m}{2}}\phi_{0}\|_{\mathbf{Y}_{0,0}}^{2}.
		\end{aligned}
	\end{equation}
	Taking $k=0$ in \eqref{33} and using \eqref{34}, one has
	\begin{equation}
		\begin{aligned}
			L^{m}\int  \phi^2
			+\int_{0}^{L}\int  X^{m} u\phi_{\xi}^2
			+ \int_{0}^{L} \int  X^{m}A \phi^2
			\leq C_{m,a} \|(b_{m,a}+\xi)^{\frac{m}{2}}\phi_{0}\|_{\mathbf{Y}_{0,0}}^{2},\quad \mbox{for}\,\, L>0.
		\end{aligned}
	\end{equation}
	Therefore the proof of Lemma \ref{RATE3} is completed.
\end{proof}

\begin{lem}\label{RATE4}	Under the assumptions of Theorem \ref{Main2},  for all integers $0\le k\le m$, it holds
	\begin{align}\label{41}
			&\frac{1}{2}\frac{d}{d X}\int \frac{\phi^2}{\bar{u}^{2}}(b_{m,a}+\xi)^k + \frac{ka}{6}\int \frac{\phi^2}{\bar{u}^2}(b_{m,a}+\xi)^{k-1} +  \frac{1}{12}  \int \frac{\phi_{\xi}^2}{\bar{u}}(b_{m,a}+\xi)^k\nonumber\\
			\leq& C_{m,a}  \int A\phi^2(1+\xi)^{k}    .
	\end{align}
\end{lem}
\begin{proof}
	Multiplying \eqref{eqphi} by $\dfrac{(b_{m,a}+\xi)^k \phi}{\bar{u}^2}$, one obtains
	\begin{equation}\label{42}
		\begin{aligned}
			&\frac{1}{2}\frac{d}{d X}\int \frac{\phi^2}{\bar{u}^{2}}(b_{m,a}+\xi)^k
			+ \frac{ka}{2}\int \frac{\phi^2}{\bar{u}^2}(b_{m,a}+\xi)^{k-1}
			+\int \frac{\phi_{\xi}^2}{\bar{u}}(b_{m,a}+\xi)^k \\
			=& -\int  \left[  1- \frac{1}{1+\frac{\epsilon\rho}{2\bar{u}}}  \right]\frac{\bar{u}_{yy}\phi^2}{\bar{u}^{4}}(b_{m,a}+\xi)^k
			+\epsilon \int\rho\frac{\phi_{\xi\xi}\phi}{\bar{u}^{2}}(b_{m,a}+\xi)^k
			-\int\frac{k}{\bar{u}}\phi\phi_\xi (b_{m,a}+\xi)^{k-1} \\
			&-\int\phi_\xi \phi \left(\frac{1}{\bar{u}}   \right)_\xi (b_{m,a}+\xi)^k.
		\end{aligned}
	\end{equation}
		For the 3rd term on RHS of \eqref{42}, one has
	\begin{equation}\label{43}
		\begin{aligned}
			\bigg|\int\frac{k}{\bar{u}}\phi\phi_\xi (b_{m,a}+\xi)^{k-1}\bigg|\leq  \frac{ka}{4} \int \frac{\phi^2}{\bar{u}^2}(b_{m,a}+\xi)^{k-1}
			+ \frac{m}{ab_{m,a}}\int \phi_{\xi}^2(b_{m,a}+\xi)^{k} .
		\end{aligned}
	\end{equation}
	For the 4th term on RHS of \eqref{42}, using  Lemma \ref{Lower base}, one obtains
	\begin{align}\label{3.17-1}
			& \Big| \int\phi_\xi \phi \left(\frac{1}{\bar{u}}\right)_\xi (b_{m,a}+\xi)^k \Big|\nonumber \\
			\leq& \left\|\frac{\phi_\xi}{\bar{u}^{\frac{1}{2}}} (b_{m,a}+\xi)^{\frac{k}{2}} \right\|_{L_\xi^2}
			\left\|\frac{\phi\bar{u}_y\chi_\frac{\delta}{a}  (b_{m,a}+\xi)^{\frac{k}{2}} }{\bar{u}^{\frac{5}{2}}}\right\|_{L_\xi^2}
			+\delta\left\|\frac{\phi_\xi (b_{m,a}+\xi)^{\frac{k}{2}}}{\bar{u}^{\frac{1}{2}}}\right\|_{L_\xi^2}^2
				+C(\frac{1}{\delta})\int A\phi^2(b_{m,a}+\xi)^{k} \nonumber\\
			\leq& \frac{4}{3C_2(\delta)^2}  \left\|\frac{\phi_\xi}{\bar{u}^{\frac{1}{2}}} (b_{m,a}+\xi)^{\frac{k}{2}} \right\|_{L_\xi^2}
			\Biggr\{
			\Big\|  \frac{\bar{u}_y \phi (b_{m,a}+\xi)^{\frac{k}{2}}   \chi_{\frac{\delta}{a}} }{\bar{u}^{\frac{3}{2}}}\Big\|_{L_{\xi}^2}
			+\Big\|  \frac{ \phi_\xi(b_{m,a}+\xi)^{\frac{k}{2}}  \chi_{\frac{\delta}{a}}}{\bar{u}^{\frac{1}{2}}}\Big\|_{L_{\xi}^2} \nonumber\\
			&+ \frac{m}{2b_{m,a}} \Big\|  \frac{ \phi (b_{m,a}+\xi)^{\frac{k}{2}} \chi_{\frac{\delta}{a}}}{\bar{u}^{\frac{1}{2}}}\Big\|_{L_{\xi}^2}
			+ \frac{C}{\delta}\Big\| \frac{\bar{u}_y \phi(b_{m,a}+\xi)^{\frac{k}{2}}}{\bar{u}^{\frac{1}{2}}}   \Big\|_{L_{\xi\geq \frac{\delta}{a}}^2}
			\Biggl\}+\delta\left\|\frac{\phi_\xi (b_{m,a}+\xi)^{\frac{k}{2}}}{\bar{u}^{\frac{1}{2}}}\right\|_{L_\xi^2}^2 \nonumber\\
			&+C(\frac{1}{\delta})\int A\phi^2(b_{m,a}+\xi)^{k}.
	\end{align}
Applying Lemma \ref{Lower base} again, one gets
	\begin{equation}\label{3.17-2}
		\begin{aligned}
			&\Big\|  \frac{\bar{u}_y \phi (b_{m,a}+\xi)^{\frac{k}{2}}   \chi_{\frac{\delta}{a}} }{\bar{u}^{\frac{3}{2}}}\Big\|_{L_{\xi}^2}\\
			\lesssim& \Big\|  \frac{\bar{u}_y \phi (b_{m,a}+\xi)^{\frac{k}{2}}   \chi_{\frac{\delta}{a}} }{\bar{u}^{\frac{1}{2}}}\Big\|_{L_{\xi}^2}
			+\Big\|  \frac{ \phi_\xi(b_{m,a}+\xi)^{\frac{k}{2}}  \chi_{\frac{\delta}{a}}}{\bar{u}^{-\frac{1}{2}}}\Big\|_{L_{\xi}^2}
			+  \Big\|  \frac{ \phi (b_{m,a}+\xi)^{\frac{k}{2}-1} \chi_{\frac{\delta}{a}}}{\bar{u}^{-\frac{1}{2}}}\Big\|_{L_{\xi}^2} + \frac{C}{\delta}\Big\| \frac{\bar{u}_y \phi(b_{m,a}+\xi)^{\frac{k}{2}}}{\bar{u}^{-\frac{1}{2}}}   \Big\|_{L_{\xi\geq \frac{\delta}{a}}^2} \\
			\lesssim& \Big\|  \frac{\delta^{\frac{1}{2}} \phi_\xi(b_{m,a}+\xi)^{\frac{k}{2}}  \chi_{\frac{\delta}{a}}}{\bar{u}^{\frac{1}{2}}}\Big\|_{L_{\xi}^2}
			+ \frac{m}{b_{m,a}} \Big\|  \frac{ \phi (b_{m,a}+\xi)^{\frac{k}{2}} \chi_{\frac{\delta}{a}}}{\bar{u}^{\frac{1}{2}}}\Big\|_{L_{\xi}^2}
			+ \frac{1}{\delta}\Big\| \frac{\bar{u}_y \phi(b_{m,a}+\xi)^{\frac{k}{2}}}{\bar{u}^{\frac{1}{2}}}   \Big\|_{L_{\xi}^2},
		\end{aligned}
	\end{equation}
where we have used the fact $C_2(\delta)\sim 1$ for $\delta\ll 1$ and $\bar{u}\sim \sqrt{a\xi}$ for $\xi\leq \frac{1}{a}$. Substituting \eqref{3.17-2} into \eqref{3.17-1}, one obtains

	\begin{align}\label{weighted F11}
			&	\Big| \int\phi_\xi \phi \left(\frac{1}{\bar{u}}\right)_\xi (b_{m,a}+\xi)^k \Big|\nonumber\\
			\leq
			& \frac{4}{3C_2(\delta)^2}  \left\|\frac{\phi_\xi}{\bar{u}^{\frac{1}{2}}} (b_{m,a}+\xi)^{\frac{k}{2}} \right\|_{L_\xi^2}
			\Biggr\{
			\left(1+C\sqrt{\delta} \right)\Big\|  \frac{ \phi_\xi(b_{m,a}+\xi)^{\frac{k}{2}}  \chi_{\frac{\delta}{a}}}{\bar{u}^{\frac{1}{2}}}\Big\|_{L_{\xi}^2}+ C\frac{m}{b_{m,a}} \Big\|  \frac{ \phi (b_{m,a}+\xi)^{\frac{k}{2}} \chi_{\frac{\delta}{a}}}{\bar{u}^{\frac{1}{2}}}\Big\|_{L_{\xi}^2} \nonumber\\
			&
			+ \frac{C}{\delta}\Big\| \frac{\bar{u}_y \phi(b_{m,a}+\xi)^{\frac{k}{2}}}{\bar{u}^{\frac{1}{2}}}   \Big\|_{L_{\xi\geq \frac{\delta}{a}}^2}
			\Biggl\} +\delta\left\|\frac{\phi_\xi (b_{m,a}+\xi)^{\frac{k}{2}}}{\bar{u}^{\frac{1}{2}}}\right\|_{L_\xi^2}^2 
			+C(\frac{1}{\delta})\int A\phi^2(b_{m,a}+\xi)^{k} \nonumber\\
			\leq& \left[\frac{4}{3C_2(\delta)^2} \left(1+C\sqrt{\delta} \right) + \delta \right]
			\left\|\frac{\phi_\xi}{\bar{u}^{\frac{1}{2}}} (b_{m,a}+\xi)^{\frac{k}{2}} \right\|_{L_\xi^2}^2
			+C(\frac{1}{\delta})\frac{m^2}{b_{m,a}^2}\Big\|  \frac{ \phi (b_{m,a}+\xi)^{\frac{k}{2}} \chi_{\frac{\delta}{a}}}{\bar{u}^{\frac{1}{2}}}\Big\|_{L_{\xi}^2}^2 \nonumber\\
			&+C(\frac{1}{\delta})\int A\phi^2(b_{m,a}+\xi)^{k} \nonumber\\
			\leq& \left[\frac{4}{3C_2(\delta)^2} \left(1+C\sqrt{\delta} \right) + \delta \right]
			\left\|\frac{\phi_\xi}{\bar{u}^{\frac{1}{2}}} (b_{m,a}+\xi)^{\frac{k}{2}} \right\|_{L_\xi^2}^2 +C(\frac{1}{\delta})(1+\frac{m^2}{a^2b_{m,a}^2}) 
			\int A\phi^2(b_{m,a}+\xi)^{k} .
	\end{align}
	For the 1st term on RHS of \eqref{42}, it is clear that
	\begin{align}
			\bigg|&\int  \left[  1- \frac{1}{1+\frac{\epsilon\rho}{2\bar{u}}}  \right]\frac{\bar{u}_{yy}\phi^2}{\bar{u}^{4}}(b_{m,a}+\xi)^k\bigg|   \nonumber\\
			\lesssim& \left\|\frac{\phi|\bar{u}_{yy}|^{\frac{1}{2}}  (b_{m,a}+\xi)^{\frac{k}{2}} \chi_{\frac{\delta}{a}}}{\bar{u}^2}   \right\|_{L_\xi^2}^2
			+ C(\frac{1}{\delta})\int A\phi^2(b_{m,a}+\xi)^{k}  \nonumber \\
			\lesssim& \delta^{\frac{1}{2}}  \left\| \frac{ \phi_\xi  (b_{m,a}+\xi)^{\frac{k}{2}}\chi_{\frac{\delta}{a}}}{\bar{u}^{\frac{1}{2}}}\right\|_{L_{\xi}^2}^2
			+\big[C(\frac{1}{\delta})+ \frac{m}{ab_{m,a}}\big]^2 \int A\phi^2(b_{m,a}+\xi)^{k}
			,
	\end{align}
	where we have used similar arguments as in Lemma \ref{Lower base} to derive 
	\begin{align}\label{44}
				\left\|\frac{\phi|\bar{u}_{yy}|^{\frac{1}{2}}  (b_{m,a}+\xi)^{\frac{k}{2}} \chi_{\frac{\delta}{a}}}{\bar{u}^2}   \right\|_{L_\xi^2}
			\lesssim& \left\|  \frac{|\bar{u}_{yy}|^{\frac{1}{2}} \phi   (b_{m,a}+\xi)^{\frac{k}{2}}\chi_{\frac{\delta}{a}}}{\bar{u}}\right\|_{L_{\xi}^2}
			+\left\|  \phi_\xi  (b_{m,a}+\xi)^{\frac{k}{2}}\chi_{\frac{\delta}{a}}\right\|_{L_{\xi}^2} \nonumber\\
			&+ \frac{m}{b_{m,a}}\left\|  \phi (b_{m,a}+\xi)^{\frac{k}{2}}\chi_{\frac{\delta}{a}}\right\|_{L_{\xi}^2}
			+\frac{C}{\delta}\left\| |\bar{u}_{yy}|^{\frac{1}{2}} \phi(b_{m,a}+\xi)^{\frac{k}{2}} \right\|_{L_{\xi\ge \frac{\delta}{a}}^2}\nonumber \\
			\lesssim& \left\| \frac{\delta^{\frac{1}{4}}}{\bar{u}^{\frac{1}{2}}} \phi_\xi  (b_{m,a}+\xi)^{\frac{k}{2}}\chi_{\frac{\delta}{a}}\right\|_{L_{\xi}^2}
			+\frac{m}{ab_{m,a}} \left\|    |\bar{u}_{yy}|^{\frac{1}{2}}  \phi (b_{m,a}+\xi)^{\frac{k}{2}}\chi_{\frac{\delta}{a}}\right\|_{L_{\xi}^2}\nonumber\\
			&+ \frac{C}{\delta}   \left(\int A\phi^2(b_{m,a}+\xi)^{k}\right)^{\frac{1}{2}} \nonumber\\
			\lesssim& \delta^{\frac{1}{4}}  \left\| \frac{ \phi_\xi  (b_{m,a}+\xi)^{\frac{k}{2}}\chi_{\frac{\delta}{a}}}{\bar{u}^{\frac{1}{2}}}\right\|_{L_{\xi}^2}
			+\big(\frac{C}{\delta}+ \frac{m}{ab_{m,a}}\big) \left(\int A\phi^2(b_{m,a}+\xi)^{k}\right)^{\frac{1}{2}}.
	\end{align}
	For the second term on RHS of \eqref{42}, it is clear that
	\begin{align*}
		\epsilon \int\rho\frac{\phi_{\xi\xi}\phi}{\bar{u}^{2}}(b_{m,a}+\xi)^k
		=&-\epsilon \int \rho_\xi \frac{\phi_\xi\phi}{\bar{u}^2}(b_{m,a}+\xi)^k
		-\epsilon\int \rho \frac{\phi_{\xi}^2}{\bar{u}^2}(b_{m,a}+\xi)^k
		-k\epsilon\int\rho \frac{\phi_{\xi}\phi}{\bar{u}^2}(b_{m,a}+\xi)^{k-1} \\
		&-\epsilon \int \rho \phi_\xi \phi \left(\frac{1}{\bar{u}^2}\right)_\xi (b_{m,a}+\xi)^k.
	\end{align*}
which, together with \eqref{rhoxi} and Lemma \ref{u0simbaru}, yields that 
	\begin{align}\label{45}
			\bigg|&\epsilon \int\rho\frac{\phi_{\xi\xi}\phi}{\bar{u}^{2}}(b_{m,a}+\xi)^k\bigg|   \nonumber\\
			\lesssim &\epsilon\int \frac{\phi_{\xi}^2}{\bar{u}}(b_{m,a}+\xi)^k +\epsilon \int \bar{u}_y \frac{|\phi_\xi\phi|}{\bar{u}^3}(b_{m,a}+\xi)^k
			+k\epsilon \int \frac{|\phi_\xi\phi|}{\bar{u}}(b_{m,a}+\xi)^{k-1},
	\end{align}
where all terms on RHS of \eqref{45} can be estimated as in \eqref{43}-\eqref{weighted F11}.
	
Substituting \eqref{43}--\eqref{45} into \eqref{42} and noting $\lim_{\delta\to 0}C_2(\delta)=\sqrt{2}$, then for sufficiently small $\delta>0,\epsilon>0$ and $b_{m,a}\gg \frac{m}{a}$, one gets
\begin{align}
			&\frac{1}{2}\frac{d}{d X}\int \frac{\phi^2}{\bar{u}^{2}}(b_{m,a}+\xi)^k
			+ \frac{ka}{6}\int \frac{\phi^2}{\bar{u}^2}(b_{m,a}+\xi)^{k-1}
			+  \frac{1}{12}  \int \frac{\phi_{\xi}^2}{\bar{u}}(b_{m,a}+\xi)^k\nonumber \\
			\leq &C_{m,a} \int A\phi^2(b_{m,a}+\xi)^{k}.
\end{align}
Therefore the proof of Lemma \ref{RATE4} is completed.
\end{proof}

\begin{lem}\label{RATE5}
	Under the assumptions of Theorem \ref{Main2}, it holds
	\begin{equation}\label{50}
		\begin{aligned}
			\sup_{X \geq0}
			\int X^{m} \frac{\phi^2}{u^{2}}
			+ \int \int X^{m} \frac{\phi_{\xi}^2}{u}
			\leq   C_{m,a} \|(1+\xi)^{\frac{m}{2}}\phi_{0}\|_{\mathbf{Y}_{0,0}}^{2}.
		\end{aligned}
	\end{equation}
\end{lem}
\begin{proof}
	Integrating \eqref{41} over $[0,L]$, we obtain
	\begin{align}\label{51.5}
			&\int \frac{\phi^2}{\bar{u}^{2}}(b_{m,a}+\xi)^m
			+ \int_{0}^{L}\int \frac{\phi^2}{\bar{u}^2}(b_{m,a}+\xi)^{m-1}
			+ \int_{0}^{L}\int \frac{\phi_{\xi}^2}{\bar{u}}(b_{m,a}+\xi)^m\nonumber\\
			\leq& C_{m,a} \|(b_{m,a}+\xi)^{\frac{m}{2}}\phi_{0}\|_{\mathbf{Y}_{0,0}}^{2}.
	\end{align}
	For any integer $0\le k \le m-1$, we multiply \eqref{41} by $X^{m-k}$ to derive
	\begin{align}\label{52}
			&\frac{1}{2}\frac{d}{d X}\int  X^{m-k}\frac{\phi^2}{\bar{u}^2}(b_{m,a}+\xi)^k
			+\frac{ka}{6}\int  X^{m-k}\frac{\phi^2}{\bar{u}^2}(b_{m,a}+\xi)^{k-1}
			+\frac{1}{12}\int  X^{m-k} \frac{\phi_{\xi}^2}{\bar{u}}(b_{m,a}+\xi)^k\nonumber \\
			\leq &C_{m,a}\int  X^{m-k}A \phi^2 (b_{m,a}+\xi)^k
			+ \frac{m-k}{2}\int X^{m-k-1}\frac{\phi^2}{\bar{u}^2}(b_{m,a}+\xi)^k.
	\end{align}
	Integrating \eqref{52} over $[0,L]$, one gets
	\begin{align}\label{52.5}
			&	\frac{1}{2}\int  L^{m-k}\frac{\phi^2}{\bar{u}^2}(b_{m,a}+\xi)^k
			+ \frac{ka}{6} \int_{0}^{L}\int  X^{m-k}\frac{\phi^2}{\bar{u}^2}(b_{m,a}+\xi)^{k-1}\nonumber\\
			&			+\frac{1}{12}\int_{0}^{L}\int  X^{m-k} \frac{\phi_{\xi}^2}{\bar{u}}(b_{m,a}+\xi)^k\nonumber \\
			\leq&  
			C_{m,a}\biggr\{\int_{0}^{L}\int X^{m-k-1}\frac{\phi^2}{\bar{u}^2}(b_{m,a}+\xi)^k
			+ \int_{0}^{L}\int  X^{m-k}A \phi^2 (b_{m,a}+\xi)^k \biggl\}  \nonumber \\
			&+\|(b_{m,a}+\xi)^{\frac{m}{2}}\phi_{0}\|_{\mathbf{Y}_{0,0}}^{2}.
	\end{align}
	From the estimates \eqref{33} and \eqref{34}, one can get
	\begin{equation}\label{53}
		\begin{aligned}
			\int_{0}^{L} \int  X^{m-k}A \phi^2 (b_{m,a}+\xi)^k
			\leq C_{m,a} \|  (b_{m,a}+\xi)^{\frac{m}{2}}\phi_{0}\|_{\mathbf{Y}_{0,0}}^{2}.
		\end{aligned}
	\end{equation}
	Thus we have for 1$\leq k \leq m-1$ that 
	\begin{equation}\label{55}
		\begin{aligned}
			\int_{0}^{L}\int X^{m-k} \frac{\phi^2}{\bar{u}^2}(b_{m,a}+\xi)^{k-1}
			\leq &C_{m,a} \biggr\{    \|(b_{m,a}+\xi)^{\frac{m}{2}}\phi_{0}\|_{\mathbf{Y}_{0,0}}^{2}
			+ \int_{0}^{L} \int X^{m-k-1} \frac{\phi^2}{\bar{u}^{2}}(b_{m,a}+\xi)^k \biggl\}
			,
		\end{aligned}
	\end{equation}which, together with \eqref{51.5}
	and 	iterating the argument for $k = m-1,m-2,\dots,1$, yields that
	\begin{equation}\label{56}
		\begin{aligned}
			\int_{0}^{L}\int X^{m-k} \frac{\phi^2}{\bar{u}^2}(b_{m,a}+\xi)^{k-1}\leq
			C_{m,a}\|(b_{m,a}+\xi)^{\frac{m}{2}}\phi_{0}\|_{\mathbf{Y}_{0,0}}^{2}.
		\end{aligned}
	\end{equation}
	Taking $k=0$ in \eqref{52.5}, then using \eqref{30}, \eqref{56} and Lemma \ref{u0simbaru}, one concludes
	\begin{equation*}
		\begin{aligned}
			\int     
			L^{m} \frac{\phi^2}{u^{2}}
			+ \int_{0}^{L} \int X^{m} \frac{\phi_{\xi}^2}{u}
			&\leq C_{m,a} \|(b_{m,a}+\xi)^{\frac{m}{2}}\phi_{0}\|_{\mathbf{Y}_{0,0}}^{2}.
		\end{aligned}
	\end{equation*}
	Therefore the proof of Lemma \ref{RATE5} is completed.
\end{proof}

\begin{lem}\label{RATE6}
	Under the assumptions of Theorem \ref{Main2}, it holds 
	\begin{align}\label{60}
		\sup_{X \geq0}
		\int X^m \phi_{\xi}^2 \,
		+  \int \int X^m \frac{\phi_{X}^2}{u}
		\leq C_{m,a}\|(1+\xi)^{\frac{m}{2}}\phi_{0}\|_{\mathbf{Y}_{0,0}}^{2}+\|\phi_{0}\|_{\mathbf{Y}_{1,0}}^2. 
	\end{align}
\end{lem}
\begin{proof}
	Multiplying \eqref{eqphi} by $\displaystyle \frac{1}{u}\phi_X X^m$, one obtains
	\begin{align}\label{61}
		\int X^m \frac{\phi_{X}^2}{u}
		+\frac{1}{2}\frac{d}{d X}\int X^m \phi_{\xi}^2
		=\frac{m}{2}\int X^{m-1} \phi_{\xi}^2
		+ a\int X^m \frac{\phi_{\xi}\phi_{X}}{u}
		-\int X^m A \phi \frac{\phi_{X}}{u}.
	\end{align}
	It is clear that
	\begin{align}\label{62}
		a\left| \int X^m \frac{\phi_{\xi}\phi_{X}}{u} \right|
		\leq \frac{1}{4} \int X^m \frac{\phi_{X}^2}{u}
		+a^2\int X^m \frac{\phi_{\xi}^2}{u}.
	\end{align}
	For the last term in \eqref{61}, similarly as \eqref{F18}, it holds that
	\begin{align}\label{63}
		\begin{aligned}
			\left| \int X^m A \phi \frac{\phi_X}{u} \right|
			\lesssim \delta \int X^m \frac{\phi_X^2}{u}
			+C(\frac{1}{\delta})a^2\int (1+X)^m \frac{\phi_\xi^2}{u}
			+C(\frac{1}{\delta})a^2\int X^m A\phi^2.
		\end{aligned}
	\end{align}
	Substituting \eqref{62}-\eqref{63} to \eqref{61}, one gets
	\begin{align}
		X^m\int \frac{\phi_{X}^2}{u}+\frac{d}{d X}\int X^m \phi_{\xi}^2\leq C_{m,a}\int (1+X)^m \frac{\phi_\xi^2}{u}
		+C_{m,a}\int X^m A\phi^2.
	\end{align}
	Then integrating the above inequality  over  $[0,L]$, and using   \eqref{30}, \eqref{50} and Lemma \ref{u0simbaru}, one concludes \eqref{60}.
	Therefore  the proof of Lemma \ref{RATE6} is completed.
\end{proof}

\begin{proof}[\textbf{Proof of Theorem \ref{Main2}}]
Using Lemma \ref{baru}, Lemma \ref{RATE3} and Lemma \ref{RATE6},
together with the Sobolev embedding, we deduce that
\begin{align}\label{decay1}
		\bigg| \frac{u\big(X,y(X,\xi;u)\big) - \bar{u}\big(y(\xi;\bar{u})\big)}{\epsilon} \bigg|
		\lesssim& \frac{|\phi(X,\xi)|}{\sqrt{a\xi}} + |\phi(X,\xi)| 
		\lesssim \frac{1}{\sqrt{a}} \| \phi_\xi\|_{L_\xi^2} + \|\phi\|_{L_\xi^2}^{\frac12} \| \phi_\xi\|_{L_\xi^2}^{\frac12}\nonumber \\
		\lesssim& C_{m,a}\left(\|(1+\xi)^{\frac{m}{2}}\phi_{0}\|_{\mathbf{Y}_{0,0}}+\|\phi_{0}\|_{\mathbf{Y}_{1,0}}\right)(1+X)^{-\frac{m}{2}}.
\end{align}
Therefore the proof of Theorem \ref{Main2} is completed.
\end{proof}

\subsection{Decay estimate in physical coordinates}\label{decay physical}

\begin{lem}\label{RATE10}
		Under the assumptions of Theorem \ref{Main3}, it holds that\begin{equation}\label{100}
			\begin{aligned}
					&\sup_{ X \geq 0}\Big\{(1+X)^{m}\int\frac{\phi_{X}^2}{u} \Big\}
				+\int\int(1+X)^{m} \phi_{X\xi}^2  
				\leq C_{a,m}  \left(\|(1+\xi)^{\frac{m}{2}}\phi_{0}\|_{\mathbf{Y}_{0,0}}^2+\|\phi_{0}\|_{\mathbf{Y}_{in}}^2\right).
			\end{aligned}
		\end{equation}
\end{lem}
\begin{proof}
Applying $\partial_X$ to \eqref{eqphi}, one has
	\begin{align}\label{3.38}
		\phi_{XX} - a\phi_{X\xi} - u\phi_{X\xi\xi} + A\phi_X - u_X \phi_{\xi\xi} + A_X \phi = 0.
	\end{align}
	Multiplying \eqref{3.38} by $\frac{1}{u}\phi_{X}(1+X)^m$ and integrating by parts, we obtain
	\begin{align}\label{101}
		&	\frac12 \frac{d}{d X}\int (1+X)^m\frac{\phi_{X}^2}{u}
			+ (1+X)^m\int  \phi_{X\xi}^2
			+(1+X)^m \int \frac{A}{u} \phi_X^2 \nonumber \\
			=&  (1+X)^m\int u_X  \frac{\phi_{\xi\xi} \phi_X}{u}
			-(1+X)^m\int A_X \frac{\phi\phi_{X}}{u}
			- (1+X)^m\frac{1}{2}\int \frac{u_X}{u^2}\phi_{X}^2  \nonumber\\
			&- \frac{a}{2}(1+X)^m\int \phi_{X}^2 \left(\frac{1}{u}\right)_\xi+ \frac{m}{2} (1+X)^{m-1}\int \frac{\phi_{X}^2}{u}    .
	\end{align}
In the following, we shall control the terms on RHS of \eqref{101}, the proof is divided into five steps.\\

 {\it Step 1.} For the second term on RHS of \eqref{101},
	we note
	\begin{align}\label{102}
		\epsilon \phi_{X}=2u u_X , \quad A_X= \frac{2\bar{u}_{yy}\bar{u}u_X}{\big[\bar{u}(\bar{u}+u)\big]^2}
		=\epsilon\cdot \frac{\bar{u}_{yy}\bar{u}}{\big[\bar{u}(\bar{u}+u)\big]^2} \cdot \frac{\phi_{X}}{u},
	\end{align}
which yields that
	\begin{equation*}
		\begin{aligned}
			\bigg|\int A_X \frac{\phi\phi_X}{u}\bigg|
			\lesssim  \epsilon\int \frac{|\bar{u}_{yy}| \phi_{X}^2 }{\bar{u}^3}
			= \epsilon\int \frac{|\bar{u}_{yy}| \phi_{X}^2 }{\bar{u}^3}\chi_{1/a}
			+  \epsilon\int \frac{|\bar{u}_{yy}| \phi_{X}^2 }{\bar{u}^3}  (1- \chi_{1/a}).
		\end{aligned}
	\end{equation*}
	By similar arguments as in the proof of Lemma \ref{Lower base}, and using Lemmas \ref{baru} \& \ref{u0simbaru}, we get
\begin{align}\label{102.5}
&(1+X)^{m}	\bigg|\int A_X \frac{\phi\phi_X}{u}\bigg|\nonumber\\
\lesssim & \epsilon(1+X)^{m}\left( \bigg\| \frac{|\bar{u}_{yy}|^{1/2}\phi_X }{\bar{u}^{1/2}}  \bigg\|_{L_{\xi}^2}^2
+  \| \sqrt{\bar{u}} \phi_{X\xi} \chi_{\delta/a}\|_{L_{\xi}^2}^2
+ \frac{1}{\delta} \|  \sqrt{\bar{u}} |\bar{u}_{yy}|^{1/2}\phi_{X}  \|_{L_\xi^2}^2 \right)\nonumber\\
&+  \frac{\epsilon a^2}{\delta}(1+X)^{m}\int \frac{\phi_{X}^2}{u}\nonumber\\
\lesssim & \epsilon\sqrt{\delta} (1+X)^{m}\int \phi_{X\xi}^2 + \frac{\epsilon a^2}{\delta} (1+X)^{m}\int \frac{\phi_{X}^2}{u}.
\end{align}

{\it Step 2}:	For the 1st term on RHS of \eqref{101},  and using \eqref{eqphi}, \eqref{102} and Lemma \ref{u0simbaru}, one has
	\begin{equation}\label{103}
		\begin{aligned}
			\bigg| \int u_X \frac{\phi_{X}\phi_{\xi\xi}}{u} \bigg|
			\lesssim \epsilon \int \frac{|\phi_{X}|^3}{u^3}
			+ \epsilon a\int \frac{\phi_{X}^2|\phi_{\xi}|}{u^3} + \epsilon \int A\frac{\phi_{X}^2}{u}.
		\end{aligned}
	\end{equation}
	
It is clear that 
	\begin{align}\label{500}
			\bigg|\int\frac{|\phi_{X}|^3}{u^3} (1-\chi_{1/a})  \bigg|
			&\lesssim \|\phi_{X}\|_{L_{\xi}^2}^2 \|\phi_{X}\|_{L_{\xi}^\infty}
			\lesssim \|\phi_{X}\|_{L_{\xi}^2}^{2}\|\phi_{X}\|_{L_{\xi}^2}^{1/2}\|\phi_{X\xi}\|_{L_{\xi}^2}^{1/2}\nonumber\\
			&\lesssim \|\phi_{X}\|_{L_{\xi}^2}^2 \|\phi_{X\xi}\|_{L_{\xi}^2}^2
			+ \|\phi_{X}\|_{L_{\xi}^2}^{8/3}.
	\end{align}
Using Lemma \ref{u0simbaru} and Hardy inequality, one obtains
	\begin{align}\label{103.52}
			\int \frac{|\phi_{X}|^3}{u^3}\chi_{1/a}
			&\lesssim \frac{1}{a^{3/2}} \int \frac{|\phi_{X}|^3}{\xi^{3/2}}
			\lesssim \frac{1}{a^{3/2}} \left(\int \frac{\phi_{X}^2}{\xi^2}\right)^{1/2} \left(\int \frac{\phi_{X}^4}{\xi}\right)^{1/2}\nonumber\\
			&\lesssim \frac{1}{a^{3/2}} \| \phi_{X\xi} \|_{L_{\xi}^2} \| \phi_{X} \|_{L_{\xi}^2} \bigg\| \frac{\phi_{X}}{\xi^{1/2}}\bigg\|_{L_{\xi}^{\infty}} \lesssim \frac{1}{a^{3/2}} \|\phi_{X}\|_{L_{\xi}^2} \|\phi_{X\xi}\|_{L_{\xi}^2}^2.
	\end{align}
Then it follows from  \eqref{500} and \eqref{103.52} that
	\begin{equation}\label{105}
		\begin{aligned}
			\epsilon \int \frac{|\phi_X|^3}{u^3}
			\lesssim \epsilon \left(\|\phi_{X}\|_{L_{\xi}^2}^2+ \frac{1}{a^{3/2}}\|\phi_{X}\|_{L_{\xi}^2} \right)\|\phi_{X\xi}\|_{L_{\xi}^2}^2
			+ \epsilon\|\phi_X\|_{L_{\xi}^2}^{8/3}.
		\end{aligned}
	\end{equation}
	
A direct estimate gives
	\begin{align}\label{106.1}
			a \int  \frac{\phi_{X}^2 |\phi_{\xi}|}{u^3}(1-\chi_{1/a})
			\lesssim&  a\|\phi_{X}\|_{L_{\xi}^{\infty}} \|\phi_{X}\|_{L_{\xi}^2}\|\phi_{\xi}\|_{L_{\xi}^2} 
			\lesssim a\|\phi_{X}\|_{L_{\xi}^2}^{3/2}\|\phi_{X\xi}\|_{L_{\xi}^2}^{1/2}\|\phi_{\xi}\|_{L_{\xi}^2}\nonumber\\
			\lesssim& a\|\phi_{\xi}\|_{L_{\xi}^2} \|\phi_X\|_{L_{\xi}^2}^2
			+a \|\phi_{\xi}\|_{L_{\xi}^2}\|\phi_{X\xi}\|_{L_{\xi}^2}^2.
	\end{align}
Using Lemma \ref{u0simbaru} and Hardy inequality, one obtains
	\begin{align}\label{106.2}
			a \int \frac{\phi_{X}^2 |\phi_{\xi}|}{u^3}\chi_{1/a}
			&\lesssim \frac{1}{\sqrt{a}} \int \frac{\phi_{X}^2|\phi_{\xi}|}{\xi^{3/2}}
			\lesssim \frac{1}{\sqrt{a}}\bigg\|\frac{\phi_{X}}{\xi^{1/2}}\bigg\|_{L_{\xi}^{\infty}}
			\bigg\| \frac{\phi_{X}}{\xi}\bigg\|_{L_{\xi}^2}\|\phi_{\xi}\|_{L_{\xi}^2}\nonumber\\
			&\lesssim \frac{1}{\sqrt{a}} \|\phi_\xi\|_{L_{\xi}^2}\|\phi_{X\xi}\|_{L_{\xi}^2}^2.
	\end{align}
Thus one has from  \eqref{106.1} and \eqref{106.2} that
	\begin{equation}\label{107}
		\begin{aligned}
			\epsilon a\int \frac{\phi_X^2|\phi_\xi|}{u^3}
			\lesssim \epsilon a\|\phi_{\xi}\|_{L_{\xi}^2} \|\phi_X\|_{L_{\xi}^2}^2
			+\epsilon \left(a+\frac{1}{\sqrt{a}}\right) \|\phi_{\xi}\|_{L_{\xi}^2}\|\phi_{X\xi}\|_{L_{\xi}^2}^2.
		\end{aligned}
	\end{equation}

Substituting \eqref{105} and \eqref{107} into \eqref{103}, we obtain
		\begin{align}\label{108}
		(1+X)^{m}	\bigg|\int u_X \frac{\phi_X\phi_{\xi\xi}}{u}\bigg| 
			\lesssim & \epsilon \|(1+X)^{\frac{m}{2}}\phi_{X\xi}\|_{L_{\xi}^2}^2 \biggl\{\|\phi_{X}\|_{L_{\xi}^2}^2+ \frac{1}{a^{3/2}}\|\phi_{X}\|_{L_{\xi}^2} 	+  \left(a+\frac{1}{\sqrt{a}}\right) \|\phi_{\xi}\|_{L_{\xi}^2}     \biggr\} \nonumber\\
			&+\epsilon \left(\|\phi_{X}\|_{L_{\xi}^2}^{\frac{2}{3}}+ a \|\phi_{\xi}\|_{L_{\xi}^2}\right)\|(1+X)^{\frac{m}{2}}\phi_{X}\|_{L_{\xi}^2}^2  +\epsilon\int (1+X)^m A\frac{\phi_{X}^2}{u}.
		\end{align}

{\it Step 3.} For the 3rd term on RHS of \eqref{101}, using \eqref{102} and \eqref{105}, one has
	\begin{align}\label{109}
			 (1+X)^{m}	\bigg|\int \frac{u_X}{u^2}\phi_{X}^2\bigg|
			 \lesssim & \epsilon \left(\|\phi_{X}\|_{L_{\xi}^2}^2+ \frac{1}{a^{3/2}}\|\phi_{X}\|_{L_{\xi}^2} \right)\|(1+X)^{\frac{m}{2}}\phi_{X\xi}\|_{L_{\xi}^2}^2\nonumber\\
			&+ \epsilon \|\phi_{X}\|_{L^2_{\xi}}^{2/3}\cdot \|(1+X)^{\frac{m}{2}}\phi_X\|_{L_{\xi}^2}^{2}.
	\end{align}

{\it Step 4.} For the 4th term on RHS of \eqref{101}, noting $u_\xi=\bar{u}_\xi+\epsilon \rho_\xi$, then using \eqref{rhoxi},   \eqref{102.5}, \eqref{107} and Lemma \ref{u0simbaru}, one obtains
\begin{align}\label{110}
&\frac{a}{2}(1+X)^{m}\bigg|\int \phi_{X}^2 \left(\frac{1}{u}\right)_\xi \bigg|
\lesssim (1+X)^{m} \int \frac{|\bar{u}_{yy}|\phi_{X}^2}{u^3} + \epsilon a(1+X)^{m}\int \frac{\phi_{X}^2|\phi_{\xi}|}{u^3}\nonumber\\
\lesssim&	\epsilon  a\|\phi_{\xi}\|_{L_{\xi}^2} \|(1+X)^{\frac{m}{2}}\phi_X\|_{L_{\xi}^2}^2
+ \epsilon \left(a+\frac{1}{\sqrt{a}}\right) \|\phi_{\xi}\|_{L_{\xi}^2}\|(1+X)^{\frac{m}{2}}\phi_{X\xi}\|_{L_{\xi}^2}^2\nonumber\\
&+ \sqrt{\delta}(1+X)^{m} \int \phi_{X\xi}^2 +   \frac{ a^2}{\delta} (1+X)^{m} \int \frac{\phi_{X}^2}{u}.
\end{align}

{\it Step 5.}  Substituting \eqref{102.5}, and \eqref{108}-\eqref{110} into \eqref{101} and taking  $\epsilon>0$, $\delta>0$ sufficiently small, one obtains
	\begin{align}\label{111}
			&\frac{d}{d X}\int (1+X)^{m} \frac{\phi_{X}^2}{u}
			+(1+X)^{m} \int  \phi_{X\xi}^2   \nonumber\\
			\lesssim_m &\|(1+X)^{\frac{m}{2}}\phi_{X\xi}\|_{L_{\xi}^2}^2 \biggl\{ \|\phi_{X}\|_{L_{\xi}^2}^2+ \frac{1}{a^{3/2}}\|\phi_{X}\|_{L_{\xi}^2} +  \left(a+\frac{1}{\sqrt{a}}\right) \|\phi_{\xi}\|_{L_{\xi}^2}    \biggr\} \nonumber\\
			&+ \left( \|\phi_{X}\|_{L_\xi^2}^{2/3}+a\|\phi_{\xi}\|_{L_{\xi}^2}    \right)\|(1+X)^{\frac{m}{2}}\phi_{X}\|_{L_{\xi}^2}^2 + (1+a^2)  \int (1+X)^{m}\frac{\phi_{X}^2}{u}.
	\end{align}
	Define the continuous function on $[0,\infty)$ by
\begin{align*}
	E(L)=\sup_{0\leq X \leq L} \Big\{ (1+X)^{m}\int \frac{\phi_{X}^2}{u}d\xi \Big\}
	+\int_{0}^{L}\int  (1+X)^{m}\phi_{X\xi}^2d\xi dX.
\end{align*}
Integrating \eqref{111}  over $[0,L]$ and  using Lemma \ref{RATE5} and Lemma \ref{RATE6}, one obtains
	\begin{align*}
		E(L)\lesssim C_{a,m} \biggl\{  & E(L)^2+E(L)^{\frac{3}{2}}
		+\big[ \|\phi_{0}\|_{\mathbf{Y}_{0,0}} +\|\phi_{0}\|_{\mathbf{Y}_{1,0}}\big] E(L) + \Big[\|(1+\xi)^{\frac{m}{2}}\phi_{0}\|_{\mathbf{Y}_{0,0}}^2 +\|\phi_{0}\|_{\mathbf{Y}_{1,0}}^2 \Big] E(L)^{\frac{1}{3}}\nonumber\\
		&+\|(1+\xi)^{\frac{m}{2}}\phi_{0}\|_{\mathbf{Y}_{0,0}}^3+\|\phi_{0}\|_{\mathbf{Y}_{1,0}}^3+\|(1+\xi)^{\frac{m}{2}}\phi_{0}\|_{\mathbf{Y}_{0,0}}^2+\|\phi_{0}\|_{\mathbf{Y}_{1,0}}^2
		+\|\phi_{0}\|_{\mathbf{Y}_{2,0}}^2  \biggr\},
	\end{align*}
which, together with  \eqref{pho initial} and the standard bootstrap argument, yields that
	\begin{align*}
	E(L)\leq C_{a,m}  \left(\|(1+\xi)^{\frac{m}{2}}\phi_{0}\|_{\mathbf{Y}_{0,0}}^2+ \|\phi_{0}\|^2_{\mathbf{Y}_{1,0}} +\|\phi_{0}\|_{\mathbf{Y}_{2,0}}^2\right) .
	\end{align*}
	Therefore the proof of Lemma \ref{RATE10} is completed.
\end{proof}

\begin{lem}\label{RATE11}
	Under the assumptions of Theorem \ref{Main3}, we have	
	\begin{equation}\label{110-1}
		\begin{aligned}
			\int \phi_{\xi}^4
			\leq C_{m,a}\left( \|(1+\xi)^{\frac{m}{2}}\phi_{0}\|_{\mathbf{Y}_{0,0}}^4+ \|\phi_{0}\|_{\mathbf{Y}_{in}}^4\right)(1+X)^{-2m}.
		\end{aligned}
	\end{equation}	
\end{lem}
\begin{proof}
	Multiplying  \eqref{eqphi} by $\dfrac{1}{u}\phi\phi_{\xi}^2(1+X)^{2m} $, one has
	\begin{equation}\label{1111}
		\begin{aligned}
			\frac13\int \phi_{\xi}^4(1+X)^{2m} + \int \frac{A}{u}\phi^2\phi_{\xi}^2(1+X)^{2m}
			=& -\int \frac{1}{u}\phi_X\phi\phi_{\xi}^2(1+X)^{2m}+ a\int \frac{1}{u}\phi\phi_{\xi}^3(1+X)^{2m}.
		\end{aligned}
	\end{equation}
	
Noting Lemma \ref{u0simbaru}, one gets
\begin{equation}\label{1113}
	\begin{aligned}
		\bigg|\frac{\phi}{u}(X,\xi)\bigg| \lesssim&\frac{|\phi(X,\xi)|}{\sqrt{a\xi}} + |\phi(X,\xi)| 
		\lesssim \frac{1}{\sqrt{a}}\|\phi_{\xi}\|_{L_{\xi}^2}+\|\phi\|_{L_{\xi}^2}^{\frac{1}{2}}\|\phi_\xi\|_{L_{\xi}^2}^{\frac{1}{2}} 
		\lesssim \frac{1}{\sqrt{a}}\|\phi_{\xi}\|_{L_{\xi}^2}+\sqrt{a}\|\phi\|_{L_\xi^2}.
	\end{aligned}
\end{equation}
which yields that
	\begin{align}\label{1112}
		\bigg| \int \frac{1}{u}\phi_X\phi\phi_{\xi}^2 \bigg|
		\leq \frac{1}{12} \int \phi_{\xi}^4 + C \bigg\|\frac{\phi}{u}\bigg\|_{L_{\xi}^{\infty}}^2\int \phi_X^2\leq \frac{1}{12} \int \phi_{\xi}^4 + C_a \left(\|\phi_{\xi}\|_{L_{\xi}^2}^2+\|\phi\|_{L_\xi^2}^2\right)\int \phi_X^2,
	\end{align}
%
and
	\begin{equation}\label{1115}
		\begin{aligned}
				a\bigg| \int \frac{1}{u}\phi\phi_{\xi}^3 \bigg|
			\leq& \frac{1}{12} \int \phi_{\xi}^4 + 3a^2\bigg\|\frac{\phi}{u}\bigg\|_{L_{\xi}^{\infty}}^2 \int \phi_{\xi}^2
			\leq \frac{1}{12} \int \phi_{\xi}^4+C_a \left(\|\phi_{\xi}\|_{L_{\xi}^2}^2+\|\phi\|_{L_\xi^2}^2\right)\int \phi_{\xi}^2.
		\end{aligned}
	\end{equation}
	
Substituting \eqref{1112}-\eqref{1115} into \eqref{1111} and applying Lemmas \ref{RATE5}--\ref{RATE10}, one obtains
	\begin{equation}
		\begin{aligned}
			\int \phi_{\xi}^4(1+X)^{2m}
			\lesssim& C_{m,a}\left( \|(1+\xi)^{\frac{m}{2}}\phi_{0}\|_{\mathbf{Y}_{0,0}}^4+ \|\phi_{0}\|_{\mathbf{Y}_{in}}^4\right).
		\end{aligned}
	\end{equation}	
	Therefore the proof of Lemma \ref{RATE11} is completed.
\end{proof}

\begin{proof}[\textbf{Proof of Theorem \ref{Main3}}]
For  $x,y\geq 0$, it holds that
\begin{align}\label{decay6}
		|u(x,y)-\bar{u}(y)|
		&\leq \big|u(x,y)-\bar{u}(y^*)\big| + \big|\bar{u}(y^*)-\bar{u}(y)\big| \nonumber\\
		&\leq \big|u(x,y)-\bar{u}(y^*)\big| + a e^{-a\hat{y}} |y-y^*|,
\end{align}
where
\begin{align*}
	y^*=\int_{0}^{\xi}\frac{1}{\bar{u}(y(\xi';\bar{u}))}d\xi', \quad
	\xi=\xi(x,y;u)=\int_{0}^{y}u(x,s)ds, \quad
	\hat{y}\in \big[\min\{y,y^*\},\max\{y,y^*\}\big].
\end{align*}

For the first term on RHS of \eqref{decay6}, noting $y=y(x,\xi;u)$ and $y^*=y(\xi;\bar{u})$,  and
applying Theorem \ref{Main2}, one has
\begin{equation}\label{decay7}
	\begin{aligned}
	&	\big|u(x,y)-\bar{u}(y^*)\big|\lesssim C_{m,a} \epsilon \left(\|(b_{m,a}+\xi)^{\frac{m}{2}}\phi_{0}\|_{\mathbf{Y}_{0,0}}+\|\phi_{0}\|_{\mathbf{Y}_{1,0}}\right)(1+x)^{-\frac{m}{2}}.
	\end{aligned}
\end{equation}

For the second term on RHS of \eqref{decay6}, noting
\begin{align*}
	|\phi(X,\xi)|\leq |\xi|^{\frac{3}{4}} \|\phi_\xi\|_{L_{\xi}^4},
\end{align*}
which, together with Lemma  \ref{baru}, yields that
\begin{align}\label{decay8}
			|y-y^*|
	\lesssim	& \int_{\frac{1}{a}}^{\max\{\frac{1}{a},\xi\}
	} \frac{\big|u\big(x,y(x,\xi';u)\big)-\bar{u}\big(y(\xi';\bar{u})\big)\big|}
		{\bar{u}\big(y(\xi';\bar{u})\big)\,u\big(x,y(x,\xi';u)\big)} d\xi' + \int_{0}^{\frac{1}{a}} \frac{\big|u\big(x,y(x,\xi';u)\big)-\bar{u}\big(y(\xi';\bar{u})\big)\big|}
		{\bar{u}\big(y(\xi';\bar{u})\big)\,u\big(x,y(x,\xi';u)\big)} d\xi' \nonumber \\
		\lesssim& \xi \, \big\|u\big(x,y(x,\xi';u)\big)-\bar{u}\big(y(\xi';\bar{u})\big)\big\|_{L_{\xi'}^\infty}
		+ \frac{\epsilon}{a^{\frac{3}{2}}}\int_{0}^{\frac{1}{a}} \frac{|\phi|}{\xi'^{\frac{3}{2}}} d\xi' \nonumber\\
		\lesssim& \xi \, \big\|u\big(x,y(x,\xi';u)\big)-\bar{u}\big(y(\xi';\bar{u})\big)\big\|_{L_{\xi'}^\infty}
		+ \frac{\epsilon}{a^{\frac{7}{4}}}  \|\phi_\xi\|_{L_{\xi}^4}.
	\end{align}
Using  \eqref{decay8} and Lemma \ref{equi2}, one obtains
	\begin{align}\label{decay9}
				a e^{-a\hat{y}} |y-y^*|
			&\lesssim a\xi e^{-Ca y(\xi;\bar{u})}
			\big\|u\big(x,y(x,\xi';u)\big)-\bar{u}\big(y(\xi';\bar{u})\big)\big\|_{L_{\xi'}^\infty}
			+ \frac{\epsilon}{a^{\frac{3}{4}}}  \|\phi_\xi\|_{L_{\xi}^4} \nonumber \\
			&\lesssim  \big\|u\big(x,y(x,\xi;u)\big)-\bar{u}\big(y(\xi;\bar{u})\big)\big\|_{L_{\xi}^\infty}
			+ \frac{\epsilon}{a^{\frac{3}{4}}}  \|\phi_\xi\|_{L_{\xi}^4},
	\end{align}
where we have used Lemma \ref{baru} to derive
\begin{align*}
	\xi e^{-Ca y(\xi;\bar{u})}
	\lesssim \xi e^{-C \sqrt{a\xi}} \mathbf{1}_{\xi\leq 1/a} + \xi e^{-C a\xi} \mathbf{1}_{\xi\geq 1/a}
	\lesssim \frac{1}{a}.
\end{align*}

Substituting \eqref{decay7} and \eqref{decay9} into \eqref{decay6} and using Lemma \ref{RATE11}, we obtain
\begin{align}\label{decay10}
	|u(x,y)-\bar{u}(y)|\leq \epsilon  C_{m,a}\left( \|(1+\xi)^{\frac{m}{2}}\phi_{0}\|_{\mathbf{Y}_{0,0}}+ \|\phi_{0}\|_{\mathbf{Y}_{in}}\right)(1+x)^{-\frac{m}{2}}.
\end{align}
Therefore the proof of Theorem \ref{Main3} is completed.
\end{proof}

	\section{Appendix A}
	\begin{lem}
		Under the assumptions  of Theorem \ref{Main}, it holds that
			\begin{align}\label{a3}
			|y(\xi;u_0)-y(\xi;\bar{u})|\lesssim \frac{\epsilon}{a}.
		\end{align}
	\end{lem}
	\begin{proof}
			It follows from \eqref{initial3} and \eqref{initial9} in the proof of Lemma \ref{u0simbaru} that \eqref{a3} holds.
			Therefore the proof of Lemma \ref{a3} is completed.
	\end{proof}
	
			\begin{lem}\label{equi2}
			Under the assumptions  of Theorem \ref{Main}, it holds that
			\begin{align}\label{a33}
				\quad	&y(X,\xi;u)\sim y(\xi;\bar{u}).
			\end{align}
		\end{lem}
		\begin{proof}
			We note that \begin{align*}
				y(X,\xi;u)=\int_{0}^{\xi}\frac{1}{u(X,y(X,\xi';u))}d\xi' \quad and \quad y(\xi;\bar{u})=\int_{0}^{\xi}\frac{1}{\bar{u}(y(\xi';\bar{u}))}d\xi', 
			\end{align*}then it follows from Lemma \ref{u0simbaru} that the proof of Lemma \ref{equi2} is completed.
		\end{proof}

		\begin{lem}\label{physical u0 sim baru}
			Under the assumptions \eqref{condition}  in Theorem \ref{Main}, we have \begin{align}\label{initial11}
				u_0(y)\sim \bar{u}(y) \quad for \quad y\geq 0.
			\end{align}
		\end{lem}
		\begin{proof}
			Using  \eqref{condition}, Sobolev embedding and the mean value theorem, we  have\begin{equation}\label{initial12}
				\begin{aligned}
					&|u_0(y)-\bar{u}(y)|\leq	 \|u_0-\bar{u}\|_{L_{y}^2}^{\frac{1}{2}}\|\partial_{y}(u_0-\bar{u})\|_{L_{y}^2}^{\frac{1}{2}}\lesssim \epsilon, \quad \forall  y\geq 0,\\
					&	|u_0(y)-\bar{u}(y)|\leq \|\partial_{y}(u_0-\bar{u})\|_{L_{y}^{\infty}}y
					\lesssim \epsilon ay ,  \quad \forall y\geq 0.
				\end{aligned}
			\end{equation} 
			Noting that $\bar{u}(y)\gtrsim 1$ for $y\geq \frac{1}{a}$ and $\bar{u}(y)\gtrsim a y$ for $y\leq \frac{1}{a}$, then we have from \eqref{initial12} that
			\[
			|u_0(y)-\bar{u}(y)|\lesssim \epsilon\,\bar{u}(y),\quad \forall\,y\geq 0,
			\]
			which  concludes    \eqref{initial11}. Therefore the proof of Lemma \ref{physical u0 sim baru} is completed.
		\end{proof}

		\begin{lem}\label{phi0}
			Under the assumptions \eqref{condition} in Theorem \ref{Main}, we have \begin{align}\label{a1}
						&	\|\phi_{0}\|_{\mathbf{Y}_{0,0}}^{2} \lesssim \frac{1}{a},\\
					&		\|\phi_{0}\|_{\mathbf{Y}_{1,0}}^2\lesssim a. \label{a2}
			\end{align}
		\end{lem}
		\begin{proof}
		We divide the proof into two steps.
		
{\it Step 1.}	
We first aim to prove 	\eqref{a1},	using Lemma \ref{baru} and Lemma \ref{u0simbaru}, we have 
			\begin{equation}\label{a4}
				\begin{aligned}
						\int|\phi_{0}(\xi)|^{2}\frac{1}{u_{0}^{2\lambda}}\mathrm{d}\xi
					\lesssim\int_{0}^{\frac{1}{a}}  |\phi_{0}(\xi)|^{2} \frac{1}{a^\lambda\xi^{\lambda}}\mathrm{d}\xi
					+ \int_{\frac{1}{a}}^{\infty}  |\phi_{0}(\xi)|^{2}\mathrm{d}\xi
					:=\frac{1}{a^\lambda}I_1 +I_2.
				\end{aligned}
			\end{equation}
			\textit{Estimate of $I_1$}:
			It is clear that
			\begin{align}\label{a5}
					\epsilon^2\int_{0}^{\frac{1}{a}}  |\phi_{0}(\xi)|^{2} \frac{1}{\xi^{\lambda}}\mathrm{d}\xi
					\lesssim& \int_{0}^{\frac{1}{a}}  \big|u_{0}^2(y(\xi;u_0))-\bar{u}^2(y(\xi;u_0))\big|^{2} \frac{1}{\xi^{\lambda}}\mathrm{d}\xi  \nonumber\\
					&	+\int_{0}^{\frac{1}{a}}  \big|\bar{u}^2(y(\xi;u_0))-\bar{u}^2(y(\xi;\bar{u}))\big|^{2} \frac{1}{\xi^{\lambda}}\mathrm{d}\xi.
			\end{align}
	For the first term on RHS of \eqref{a5}, performing a change of variable and applying Lemma \ref{baru} and Lemma \ref{equi2}, one gets
			\begin{align}\label{a6}
					&\int_{0}^{\frac{1}{a}}  \big|u_{0}^2(y(\xi;u_0))-\bar{u}^2(y(\xi;u_0))\big|^{2} \frac{1}{\xi^{\lambda}}\mathrm{d}\xi  \nonumber\\
					=&\int_{0}^{y(\frac{1}{a};u_0)}  \big|u_{0}^2(y)-\bar{u}^2(y)\big|^{2} \frac{u_0(y)}{\xi(y;u_0)^{\lambda}}\mathrm{d}y  \nonumber\\
					\lesssim&\int_{0}^{M_5y(\frac{1}{a};\bar{u})}  \big|u_{0}^2(y)-\bar{u}^2(y)\big|^{2} \frac{u_0(y)}{\xi(y;u_0)^{\lambda}}\mathrm{d}y \nonumber\\
					\lesssim&\int_{0}^{\frac{M_6}{a}}  \big|u_{0}^2(y)-\bar{u}^2(y)\big|^{2} \frac{u_0(y)}{\xi(y;u_0)^{\lambda}}\mathrm{d}y.
			\end{align}
		Noting Lemma \ref{physical u0 sim baru}, we obtain
			\begin{align}\label{a7}
				\xi(y;u_0)=\int_{0}^{y}u_0(y')\mathrm{d}y'\sim \int_{0}^{y}\bar{u}(y')\mathrm{d}y'=\xi(y;\bar{u}).
			\end{align}
			Using Lemma \ref{baru}  and noting $1\leq\lambda<\frac{3}{2}$, one has
			\begin{align}\label{a8}
						\big|u_{0}^2(y)-\bar{u}^2(y)\big|^{2} \frac{u_0(y)}{\xi(y;u_0)^{\lambda}}
					&\lesssim \big|u_0(y)-\bar{u}(y)\big|^2\frac{\bar{u}^3(y)}{\xi(y;\bar{u})^{\lambda}} \nonumber\\
					&\lesssim \big|u_0(y)-\bar{u}(y)\big|^2a^{3-\lambda}y^{3-2\lambda}\nonumber \\
					&\lesssim \big|u_0(y)-\bar{u}(y)\big|^2a^\lambda, \quad for \ y\in[0, \frac{M_6}{a}].
			\end{align}
			Substituting  \eqref{a8} into \eqref{a6},	it  follows that
			\begin{align}\label{a9}
			\int_{0}^{\frac{1}{a}}  \big|u_{0}^2(y(\xi;u_0))-\bar{u}^2(y(\xi;u_0))\big|^{2} \frac{1}{\xi^{\lambda}}\mathrm{d}\xi
				\lesssim a^{\lambda} \big\|u_0-\bar{u}\big\|_{L_{y}^2}^2
				\lesssim \frac{a^{\lambda+2}}{(1+a)^4}\epsilon^2.
			\end{align}
			For the second term on RHS of \eqref{a5}, 
		we have from Lemma \ref{baru} and Lemma \ref{u0simbaru} that 
			\begin{align}\label{a10}
					&\int_{0}^{\frac{1}{a}}  \big|\bar{u}^2(y(\xi;u_0))-\bar{u}^2(y(\xi;\bar{u}))\big|^{2} \frac{1}{\xi^{\lambda}}\mathrm{d}\xi  \nonumber\\
					\lesssim& \int_{0}^{\frac{1}{a}}  \big|\bar{u}(y(\xi;u_0))+\bar{u}(y(\xi;\bar{u}))\big|^{2} \frac{1}{\xi^{\lambda}}\mathrm{d}\xi \cdot \big\|\bar{u}(y(\xi;u_0))-\bar{u}(y(\xi;\bar{u}))\big\|_{L_{\xi}^\infty}^2  \nonumber\\
					\lesssim& \int_{0}^{\frac{1}{a}}   \frac{a^3}{\xi^{\lambda-1}}\mathrm{d}\xi \cdot \big\|y(\xi;u_0)-y(\xi;\bar{u})\big\|_{L_{\xi}^\infty}^2  \nonumber\\
					\lesssim& a^{1+\lambda}  \big\|y(\xi;u_0)-y(\xi;\bar{u})\big\|_{L_{\xi}^\infty}^2 
					\lesssim a^{\lambda-1}\epsilon^2.
			\end{align}
			Combining  \eqref{a9}  and \eqref{a10}, we conclude
			\begin{align}\label{a11}
			\frac{1}{a^{\lambda}}	I_1\lesssim \frac{1}{a}.
			\end{align}
			\textit{ Estimate of $I_2$}: It is clear that
			\begin{align*}
				\epsilon^2	\int_{\frac{1}{a}}^{\infty}  |\phi_{0}(\xi)|^{2}d\xi  
				\lesssim&  \int_{\frac{1}{a}}^{\infty}  |u_{0}^2(y(\xi;u_0))-\bar{u}^2(y(\xi;u_0))|^{2} d\xi +\int_{\frac{1}{a}}^{\infty}  |\bar{u}^2(y(\xi;u_0))-\bar{u}^2(y(\xi;\bar{u}))|^{2} d\xi .
			\end{align*}
			By changing the variable once again and using the fact that $u_0$ is uniformly bounded, one has
			\begin{align*}
				\int_{\frac{1}{a}}^{\infty}  |u_{0}^2(y(\xi;u_0))-\bar{u}^2(y(\xi;u_0))|^{2} d\xi =&
				\int_{y(\frac{1}{a};u_0)}^{\infty}  |u_{0}^2(y)-\bar{u}^2(y)|^{2}u_0(y)dy\\
				\lesssim& 	\int  |u_{0}(y)-\bar{u}(y)|^{2}dy\lesssim \frac{a^2}{(1+a)^4}\epsilon^2.
			\end{align*}
		Noting $\partial_y\bar{u}= ae^{-ay}$, Lemma \ref{baru}, Lemma \ref{equi2} and \eqref{a3}, one has that
			\begin{align*}
				 \int_{\frac{1}{a}}^{\infty}  |\bar{u}^2(y(\xi;u_0))-\bar{u}^2(y(\xi;\bar{u}))|^{2} d\xi 
				&\lesssim    \int_{\frac{1}{a}}^{\infty}    |\bar{u}(y(\xi;u_0))-\bar{u}(y(\xi;\bar{u}))|^{2}  d\xi\\
				&\lesssim\int_{\frac{1}{a}}^{\infty}a^2e^{-M_7a\xi} \|y(\xi;u_0)-y(\xi;\bar{u})\|_{L_{\xi}^\infty}^2 d\xi
				\lesssim \frac{\epsilon^2}{a}.
			\end{align*}
			Thus \begin{align*}
				I_2\lesssim \frac{1}{a}.
			\end{align*}

{\it Step 2.} For \eqref{a2}, we denote 
\begin{align*}
I_3:=\int_{0}^{\frac{1}{a}} |\phi_{0}'(\xi)|^2d\xi, \quad I_4:=\int_{\frac{1}{a}}^{\infty}|\phi_{0}'(\xi)|^2d\xi
\end{align*}
where
\begin{align*}
\epsilon	\phi_{0}'(\xi)=2[(\partial_{y}u_0)(y(\xi;u_0))-(\partial_{y}\bar{u})(y(\xi;\bar{u}))].
\end{align*}
By similar argument as  in {\it Step 1}, one gets  
\begin{align*}
I_3\lesssim a, \quad I_4\lesssim a.
\end{align*}
Therefore the proof of Lemma \ref{phi0} is completed.
\end{proof}

		\begin{lem}\label{phi0m}
			Under the assumptions \eqref{condition} and   \eqref{initcond4}, we have \begin{align}\label{b0}
			\int|\phi_{0}(\xi)|^{2}\frac{\xi^m}{u_{0}^{2\lambda}}d\xi\lesssim_m \frac{1}{a^{m+1}}+ \frac{1}{\epsilon^2}\|	y^{\frac{m}{2}}(u_0-\bar{u})(y)\|_{L_{y}^2}^2   .
			\end{align}
		\end{lem}
		\begin{proof}
			In view of the proof of $I_1$  in Lemma \ref{phi0}, we have
			\begin{align}\label{b1}
				\int_{0}^{\frac{1}{a}} |\phi_{0}(\xi)|^{2} \frac{\xi^m}{u_{0}^{2\lambda}}   \,\mathrm{d}\xi\lesssim\frac{1}{a^{m+1}}.
			\end{align}
		We note
			\begin{align}\label{b2}
						\epsilon^2\int_{\frac{1}{a}}^{\infty} |\phi_{0}(\xi)|^{2}\frac{\xi^m}{u_{0}^{2\lambda}}   \,\mathrm{d}\xi
					\lesssim&
					\int_{\frac{1}{a}}^{\infty} \big|u_{0}^2(y(\xi;u_0))-\bar{u}^2(y(\xi;u_0))\big|^{2}\xi^m \,\mathrm{d}\xi\nonumber\\
					&+\int_{\frac{1}{a}}^{\infty} \big|\bar{u}^2(y(\xi;u_0))-\bar{u}^2(y(\xi;\bar{u}))\big|^{2}\xi^m \,\mathrm{d}\xi.
			\end{align}
			For the first term on RHS of \eqref{b2}, performing a change of variable and  	applying \eqref{a7} and Lemma \ref{baru}, one gets
			\begin{align}\label{b3}
				\int_{\frac{1}{a}}^{\infty} \big|u_{0}^2(y(\xi;u_0))-\bar{u}^2(y(\xi;u_0))\big|^{2}\xi^m \,\mathrm{d}\xi
				&\lesssim \int \big|u_{0}(y)-\bar{u}(y)\big|^{2}y^m \,\mathrm{d}y.
			\end{align}
		For the second term on RHS of \eqref{b2}, 	applying Lemma \ref{baru}, Lemma \ref{equi2} and \eqref{a3}, one has that
			\begin{align}\label{b4}
					 \int_{\frac{1}{a}}^{\infty} \big|\bar{u}^2(y(\xi;u_0))-\bar{u}^2(y(\xi;\bar{u}))\big|^{2} \xi^m \,\mathrm{d}\xi 
					\lesssim& \int_{\frac{1}{a}}^{\infty} \big|\bar{u}(y(\xi;u_0))-\bar{u}(y(\xi;\bar{u}))\big|^{2} \xi^m \,\mathrm{d}\xi \nonumber\\
					\lesssim& \int_{\frac{1}{a}}^{\infty}
					a^2 e^{-M_7a\xi} \big|y(\xi;u_0)-y(\xi;\bar{u})\big|^2 \xi^m \,\mathrm{d}\xi
					\lesssim_m \frac{\epsilon^2}{a^{m+1}}.
			\end{align}
				Combining \eqref{b1}, \eqref{b3} and \eqref{b4} yields \eqref{b0}.
					Therefore the proof of Lemma \ref{phi0m} is completed.
		\end{proof}

		\begin{lem}\label{F2,0}
	For any sufficiently small $\delta_{a,m}\ll 1$, there exists a constant $\eta_{a,m}>0$ such that
	if \eqref{initcond5} holds, then
	\begin{equation}\label{c0}
		\|(1+\xi)^{\frac{m}{2}}\phi_{0}\|_{\mathbf{Y}_{0,0}}+\|\phi_{0}\|_{\mathbf{Y}_{\mathrm{in}}}
		\leq \delta_{a,m}.
	\end{equation}
		\end{lem}
		\begin{proof}
			1. We fix $\epsilon$ and replace $\epsilon$ with $\epsilon\kappa_{a,m}$, such that  
				\begin{align}\label{c1}
				\| u_0 - \bar{u} \|_{L_{y}^1}
				+  \big\| u_0 - \bar{u} \big\|_{H_{y}^2}
				+ \big\| y^{\frac{m}{2}}( u_0 - \bar{u}) \big\|_{L_{y}^2}
				\leq   \frac{a}{(1+a)^2}\epsilon\kappa_{a,m}.
			\end{align}
			Then it follows from Lemma \ref{a3}, Lemma \ref{phi0} and Lemma \ref{phi0m}  that 
			\begin{equation}\label{c2}
				\begin{aligned}
				&	|y(\xi;u_0)-y(\xi;\bar{u})|\lesssim \frac{\epsilon\kappa_{a,m}}{a},\\
				&	\|\phi_{0}\|_{\mathbf{Y}_{0,0}}^{2} \lesssim \frac{\kappa_{a,m}^2}{a},\\
				&	\|\phi_{0}\|_{\mathbf{Y}_{1,0}}^2\lesssim a\kappa_{a,m}^2,\\
			&\int|\phi_{0}(\xi)|^{2}\frac{\xi^m}{u_{0}^{2\lambda}}d\xi\lesssim C_{a,m}\kappa_{a,m}^2.
				\end{aligned}
			\end{equation}
			
2. By the same argument as in the proof of Lemma \ref{phi0}, we obtain
			\begin{align*}
				\|\phi_{0}\|_{\mathbf{Y}_{2,0}}^{2}
				&=\int_{0}^{+\infty}\frac{|\phi_{X}(0,\xi)|^2}{u_0}\, d\xi
				\leq C_{a}\kappa_{a,m}^2,
			\end{align*}
			where
			\begin{equation}
				\begin{aligned}
					\epsilon\phi_{X}(0,\xi)
					&= \bigl(2a u_{0y}\big(y(\xi;u_0)\big) + 2u_{0yy}\big(y(\xi;u_0)\big)\bigr) \\
					&\quad - \bigl(2a \bar{u}_{y}\big(y(\xi;\bar{u})\big) + 2\bar{u}_{yy}\big(y(\xi;\bar{u})\big)\bigr).
				\end{aligned}
			\end{equation}
			
			Taking $\kappa_{a,m}$ sufficiently small, we conclude \eqref{c0}.
				Therefore the proof of Lemma \ref{F2,0} is completed.
		\end{proof}

	\noindent{\bf Acknowledgments.}
	Yong Wang's research is partially supported by the National Key Research and Development Program of China, grant 2021YFA1000800; the National Natural Science Foundation of China, grants 12288201 and 12421001; and the CAS Project for Young Scientists in Basic Research, grant YSBR-031. 
	
	\medskip
	
	\noindent{\bf Conflict of Interest:} The authors declare that they have no conflict of interest. 
	
	\bigskip


\end{document}